\documentclass[11pt]{amsart}
\pdfoutput=1
\usepackage{latexsym,amscd,amssymb, graphicx, shuffle}  
\usepackage[margin=1.0in]{geometry}

\numberwithin{equation}{section}
\newtheorem{thm}{Theorem}[section]
\newtheorem{prop}[thm]{Proposition}
\newtheorem{cor}[thm]{Corollary}
\newtheorem{lem}[thm]{Lemma}
\newtheorem{conj}[thm]{Conjecture}

\newtheorem{problem}[thm]{Problem}

\newtheorem{defn}[thm]{Definition}
\theoremstyle{definition}

\newcommand{\PP}{{\mathbb{P}}}

\newcommand{\NN}{{\mathbb{N}}}
\newcommand{\CC}{{\mathbb{C}}}
\newcommand{\QQ}{{\mathbb{Q}}}
\newcommand{\ZZ}{{\mathbb{Z}}}
\newcommand{\RR}{{\mathbb{R}}}
\newcommand{\Ass}{{\mathsf{Ass}}}
\newcommand{\Asshat}{{\widehat{\mathsf{Ass}}}}
\newcommand{\Cat}{{\mathsf{Cat}}}
\newcommand{\Nar}{{\mathsf{Nar}}}
\newcommand{\Kirk}{{\mathsf{Kirk}}}
\newcommand{\Krew}{{\mathsf{Krew}}}
\newcommand{\sym}{{\mathfrak{S}}}

\begin{document}

\date{May 2013}

\title[Rational associahedra and noncrossing partitions]
{Rational associahedra and noncrossing partitions}

\author{Drew Armstrong}
\address{Dept. of Mathematics\\University of Miami\\
Coral Gables, FL, 33146}
\email{armstrong@math.miami.edu}

\author{Brendon Rhoades}
\address{Dept. of Mathematics\\University of California - San Diego\\
La Jolla, CA, 92093}
\email{bprhoades@math.ucsd.edu}

\author{Nathan Williams}
\address{Dept. of Mathematics\\University of Minnesota\\
Minneapolis, MN, 55455}
\email{will3089@math.umn.edu}



\begin{abstract}

Each positive rational number $x>0$ can be written {\bf uniquely} as $x=a/(b-a)$ for coprime 
positive integers $0<a<b$. We will identify $x$ with the pair $(a,b)$. In this paper we define for 
each positive rational $x>0$ a simplicial complex $\Ass(x)=\Ass(a,b)$ called the
 {\sf rational associahedron}. It is a pure simplicial complex of dimension $a-2$, 
 and its maximal faces are counted by the {\sf rational Catalan number} 
 $$\Cat(x)=\Cat(a,b):=\frac{(a+b-1)!}{a!\,b!}.$$ The cases $(a,b)=(n,n+1)$ and
  $(a,b)=(n,kn+1)$ recover the classical associahedron and its ``Fuss-Catalan" 
  generalization studied by Athanasiadis-Tzanaki and Fomin-Reading. 
  We prove that $\Ass(a,b)$ is shellable and give nice product formulas for its $h$-vector 
  (the {\sf rational Narayana numbers}) and $f$-vector (the {\sf rational Kirkman numbers}). 
We define $\Ass(a,b)$ via {\sf rational Dyck paths}: 
lattice paths from $(0,0)$ to $(b,a)$ staying above the line
$y = \frac{a}{b}x$. 
We also use rational Dyck paths to define a rational generalization
of noncrossing perfect matchings of $[2n]$.
In the case $(a,b) = (n, mn+1)$, our construction produces the noncrossing partitions
of $[(m+1)n]$ in which each block has size $m+1$.
\end{abstract}
\maketitle

\section{Motivation}
\label{Motivation}

The classical Catalan numbers $\Cat(n)$ are parametrized by a positive
integer $n$.  In this paper
we will study a Catalan number $\Cat(x) \in \ZZ_{> 0}$ defined 
for every positive rational number $x$ which
agrees with the classical Catalan number when $x$ is an integer.  
These `rational Catalan numbers' have additional number theoretic structure coming
from the Euclidean algorithm.
In this paper we initiate the systematic study of 
{\it rational Catalan combinatorics} by generalizing Dyck paths, the associahedron, 
noncrossing perfect matchings, and noncrossing partitions to this rational setting.
These rational generalizations are further generalizations of the so-called `Fuss analogs'
and share many of the nice combinatorial properties of their classical counterparts.
In a companion paper \cite{ArmstrongLoehrWarrington}, the first author, Loehr, and Warrington
will develop rational analogs of parking functions and their associated $q,t$-statistics.

The {\sf classical Catalan numbers}\footnote{This notation will be justified shortly.} 
$$\Cat(n,n+1)=\frac{1}{n+1}\binom{2n}{n}$$
are among the most important sequences in combinatorics. 
As of this writing, they are known to count at least 201 distinct families of combinatorial
 objects \cite{Stanley}. For our current purpose, the following three are the most important:
\begin{enumerate}
	\item Dyck paths from $(0,0)$ to $(n,n)$,
	\item Triangulations of a convex $(n+3)$-gon, and
	\item Noncrossing partitions of $[n] := \{1, 2, \dots, n\}$.
\end{enumerate}

There are two observations that have spurred recent progress in this field. The first is that Catalan objects are revealed to be ``type $A$" phenomena (corresponding to the symmetric group) when properly interpreted in the context of reflection groups. The second is that many definitions of Catalan objects can be further generalized to accommodate an additional parameter, so that the resulting objects are counted by Fuss-Catalan numbers (see \cite[Chapter 5]{Armstrong}).

Both of these generalizations can be motivated from Garsia's and Haiman's~\cite{GarsiaHaiman} observation that the Catalan numbers play a deep role in representation theory.  The symmetric group 
$\sym_n$ acts on the polynomial ring $DS_n:=\QQ[x_1,\ldots,x_n,y_1,\ldots,y_n]$ by permuting variables ``diagonally." That is, for $w\in\sym_n$ we define $w.x_i=x_{w(i)}$ and $w.y_i=y_{w(i)}$.  
Weyl \cite{Weyl} proved that the subring of ``diagonal invariants" is generated by the polarized power sums $p_{r,s}=\sum_i x_i^ry_i^s$ for $r+s\geq 0$ with $1\leq r+s\leq n$.  
The quotient ring of ``diagonal coinvariants" $DR_n:=DS_n/(p_{r,s})$ inherits the structure of an
$\sym_n$-module which is ``bigraded" by $x$-degree and $y$-degree.  
Garsia and Haiman conjectured that $\dim DR_n=(n+1)^{n-1}$ (a number famous from Cayley's formula \cite{Cayley}) and that the dimension of the sign-isotypic component is the Catalan number 
$\Cat(n,n+1)$.  
These conjectures turned out to be difficult to resolve, and were proved about ten years later by Haiman using the geometry of Hilbert schemes.

An excellent introduction to this subject is Haiman's paper~\cite{HaimanConjectures}, in 
which he laid the foundation for generalizing the theory of diagonal coinvariants 
to other reflection groups.  
Let $W$ be a Weyl group, so that $W$ acts irreducibly on $\RR^\ell$ by reflections 
and stabilizes a full-rank lattice $\ZZ^\ell\approx Q\subseteq \RR^\ell$, called the {\sf root lattice}. 
The group also comes equipped with special integers $d_1\leq \cdots \leq d_\ell$ called 
{\sf degrees}, of which the largest $h:=d_\ell$ is called the {\sf Coxeter number}. 
Haiman showed that the number of orbits of $W$ acting on the ``finite torus" $Q/(h+1)Q$  is equal to

\begin{equation*}
\Cat(W):=\prod_i \frac{h+d_i}{d_i},
\end{equation*}
which we now refer to as the {\sf Catalan number} of $W$. 

From this modern perspective, our three examples above become:

\begin{enumerate}
	\item $W$-orbits of the finite torus $Q/(h+1)Q$~\cite{ShiOne, HaimanConjectures, AthanasiadisCharacteristic, CelliniPapi},%
	\item Clusters in Fomin and Zelevinsky's finite type cluster algebras~\cite{FominZelevinsky}, and
	\item Elements beneath a Coxeter element $c$ in the absolute order on $W$~\cite{Reiner, Armstrong}.
\end{enumerate}



More generally, given any positive integer 
$p$ {\bf coprime} to the Coxeter number $h$, Haiman showed that the number of orbits of $W$ acting on the finite torus $Q/pQ$ is equal to
\begin{equation}
\label{eq:rationalcat}
\Cat(W,p):=\prod_i \frac{p+d_i-1}{d_i},
\end{equation}
which we now refer to as a {\sf rational Catalan number}.

The cases $p=mh+1$ have been extensively studied as the ``Fuss-analogues'', which further generalize our initial three examples to: 

\begin{enumerate}
	\item Dominant regions in the $m$-Shi arrangement~\cite{Athanasiadis, FishelVaziraniTwo},
	\item Clusters in the generalized cluster complex~\cite{FominReading}, and  
	\item $m$-multichains in the noncrossing partition lattice.~\cite{Edelman, Armstrong}.
\end{enumerate}

The broad purpose of ``rational Catalan combinatorics" is to complete the generalization from $p=+1\bmod{h}$ to all parameters $p$ coprime to $h$. 
That is, we wish to define and study Catalan objects such as parking functions, Dyck paths, triangulations, and noncrossing partitions for each pair $(W,p)$, where $W$ is a finite reflection group and $p$ is a positive integer coprime to the Coxeter number $h$.  We may think of this as a two-dimensional problem with a ``type axis" $W$ and a ``parameter axis" $p$. The level set $p=h+1$ is understood fairly well, and the ``Fuss-Catalan" cases $p=+1\bmod{h}$ are discussed in Chapter 5 of Armstrong \cite{Armstrong}. However, it is surprising that the type $A$ level set (i.e. $W=\sym_n$) is an open problem. This could have been pursued fifty years ago, but no one has done so in a systematic way.

Thus, we propose to begin the study of ``rational Catalan combinatorics" with the study of ``classical rational Catalan combinatorics" corresponding to a pair $(\sym_a,b)$ with $b$ coprime to $a$. In this case we have the classical rational Catalan number
\begin{equation}
\label{eq:typeArationalcat}
\Cat(\sym_a,b)=\frac{1}{a+b}\binom{a+b}{a,b}=\frac{(a+b-1)!}{a!\, b!}.
\end{equation}
Note the surprising symmetry between $a$ and $b$; i.e. that $\Cat(\sym_a,b)=\Cat(\sym_b,a)$. This will show up as a conjectural
Alexander duality in our study of rational associahedra. 

First we will set down notation for the rational Catalan numbers $\Cat(\sym_a,b)$ in Section \ref{Rational Catalan Numbers}. Then in Section \ref{Rational Dyck Paths} we will define the ``rational Dyck paths" which are the heart of the theory. In Section \ref{Rational Associahedra} we will use the Dyck paths to define and study ``rational associahedra." 
The project of generalizing these constructions to reflection groups beyond $\sym_n$ is left for the future.

\section{Rational Catalan Numbers}
\label{Rational Catalan Numbers}

Given a rational number $x\in\QQ$ outside the range $[-1,0]$, note that there is a unique way to write $x=a/(b-a)$ where $a\neq b$ are coprime positive integers. 
We will identify $x\in\QQ$ with the ordered pair $(a,b)\in\NN^2$ when convenient.

Inspired by the formulas \eqref{eq:rationalcat} and \eqref{eq:typeArationalcat} above, we define the 
{\sf rational Catalan number}:
\begin{equation*}
\Cat(x)=\Cat(a,b):=\frac{1}{a+b}\binom{a+b}{a,b}=\frac{(a+b-1)!}{a!\, b!}.
\end{equation*}
Note that this formula is symmetric in $a$ and $b$. This, together with the fact that $a/(b-a)=x$ if and only if $b/(a-b)=-x-1$, gives us
\begin{equation*}
\Cat(x)=\Cat(a,b)=\Cat(b,a)=\Cat(-x-1).
\end{equation*}
That is, the function $\Cat:\QQ\setminus[-1,0]\to\NN$ is symmetric about $x=-1/2$. Now observe that $-\frac{1}{x-1}-1=\frac{x}{1-x}$, and hence $\Cat(1/(x-1))=\Cat(x/(1-x))$. We call this value the 
{\sf derived Catalan number}:
\begin{equation*}
\Cat'(x):=\Cat(1/(x-1))=\Cat(x/(1-x)).
\end{equation*}
Furthermore, note that $\frac{1}{1/x-1}=\frac{x}{1-x}$, hence
\begin{equation}
\label{eq:rationalduality}
\Cat'(x)=\Cat'(1/x).
\end{equation}
We call this equation {\sf rational duality} and it will play an important role in 
our study of rational associahedra below. Equation \eqref{eq:rationalduality} 
can also be used to extend the domain of 
$\Cat'$ from $\QQ\setminus[-1,0]$ to $\QQ\setminus\{0\}$, but we don't know if this holds
combinatorial significance.
In terms of $a$ and $b$ we can write
\begin{equation*}
\Cat'(x)=\Cat'(a,b)=\begin{cases} \binom{b}{a}/b & \text{ if } a<b,\\ \binom{a}{b}/a & \text{ if } b<a. \end{cases}
\end{equation*}
The ``derivation" of Catalan numbers can be viewed as a ``categorification" of the Euclidean algorithm. For example, consider $x=5/3$ (that is, $a=5$ and $b=8$).  The continued fraction expansion of $x$ is
\begin{equation*}
  \frac{5}{3} = 1 + \cfrac{1}{1
          + \cfrac{1}{1
          + \cfrac{1}{1}}}
\bigskip
\end{equation*}
with ``convergents" (that is, successive truncations) $\frac{1}{1},\frac{2}{1},\frac{3}{2},\frac{5}{3}$. Thus we have
\begin{align*}
\Cat(5/3) &=99,\\
\Cat'(5/3) &= \Cat(3/2) = 7,\\
\Cat''(5/3) &= \Cat'(3/2) = \Cat(2) = 2,\\
\Cat'''(5/3) &= \Cat''(3/2) = \Cat'(2) = \Cat(1) = 1.
\end{align*} 
The process stabilizes because $\Cat'(1)=1$. Finally, we observe the most important feature of the rational Catalan numbers. They are backwards-compatible:
\begin{equation*}
\Cat(n)=\Cat(n/1)=\Cat(n,n+1)=\frac{1}{2n+1}\binom{2n+1}{n,n+1}=\frac{1}{n+1}\binom{2n}{n}.
\end{equation*}

\section{Rational Dyck Paths}
\label{Rational Dyck Paths}


At the heart of our constructions lies a family of lattice paths called ``rational Dyck paths".  We motivate
their definition with results from the theory of Weyl groups.

\subsection{Weyl Groups}
\label{Weyl Groups}
Let $W$ be a Weyl group with {\sf root system} $\Phi$ and
 {\sf simple roots} $\Pi\subseteq\Phi$ (so that $\Phi=W.\Pi$, and every element 
 of $\Phi$ is a non-negative or non-positive $\ZZ$-linear comination of simple roots). 
 Let $\Phi_k\subseteq \Phi$ be the set of roots of ``height $k$," in which the $\Pi$-coefficients 
 sum to $k$. It is true that $\Phi$ contains a unique root $\theta$ of maximum 
 height $h - 1$, where
 $h$ is the Coxeter number of 
 $W$. The group $\hat{W}$ generated by the reflections in the linear 
 hyperplanes $(\bullet,\alpha)=0$ for all $\alpha\in\Pi$ and the affine 
 hyperplane $(\bullet,\theta)=1$ is called the {\sf affine Weyl group}. The simplex
\begin{equation*}
A_\circ:=\left\{ x\in\RR^\ell: (x,\alpha)>0 \text{ for all $\alpha\in\Pi$ and } (x,\theta)<1\right\}
\end{equation*}
is a fundamental domain for $\hat{W}$, called the {\sf fundamental alcove}. More generally, let $p$ be
 coprime to $h$ and write $p=qh+r$, where $1\leq r<p$. Sommers \cite{Sommers}
  proved that the simplex
\begin{equation*}
D^p:=\left\{ x\in\RR^\ell: (x,\alpha)<q \text{ for $\alpha\in\Phi_r$ and } (x,\alpha)<q+1 \text{ for $\alpha\in\Phi_{r-h}$}\right\}
\end{equation*}
is congruent to the dilation $pA_\circ$ of the fundamental alcove, and hence $D^p$
 contains $p^\ell$ alcoves $A_w$ (corresponding to $p^\ell$ group elements $w\in\hat{W}$). 
 Furthermore, the alcoves $A_w\in D^p$ such that $A_{w^{-1}}$ is ``positive" (i.e. $(x,\alpha)>0$ for all
  $x\in A_{w^{-1}}$) are in bijection with $W$-orbits on $Q/pQ$, and hence are counted by the number $
  \Cat(W,p)$. For example, the left side of Figure \ref{fig:simplex_shi} displays the simplex $D^4$ for the 
  symmetric group $W=\sym_3$. Here we have $(\ell,h,p)=(2,3,4)$, which gives $4^2=16$ alcoves in 
  $D^4$ and $\Cat(\sym_3)=5$ positive alcoves.

\begin{figure}
\begin{center}
\includegraphics[scale=.3]{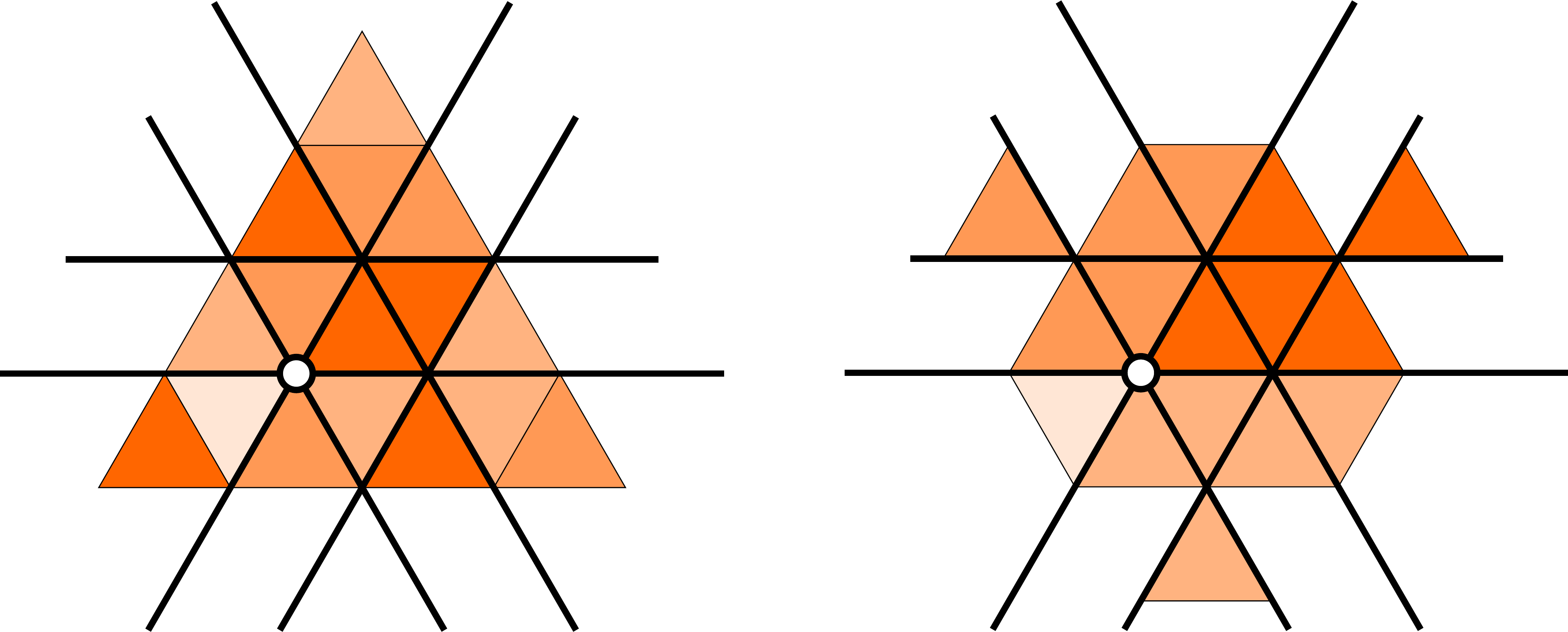}
\end{center}
\caption{The simplex $D^4$ and the Shi arrangement for $W=\sym_3$.}
\label{fig:simplex_shi}
\end{figure}

In the cases $p=\pm 1\bmod{h}$, the collection of alcoves $A_{w^{-1}}$ 
for $A_w\in D^p$ are related to the {\sf $m$-Shi hyperplane arrangement}, 
consisting of the hyperplanes
 $(\bullet,\alpha)\in\{-m+1,\ldots,m\}$ for $\alpha\in\Phi^+$ (these are the {\sf positive roots} 
 whose $\Pi$-expansions have positive coefficients). In particular, Fishel and Vazirani proved 
 that the inverses of $D^{mh+1}$ are precisely the {\bf minimal} alcoves in the chambers of 
 the $m$-Shi arrangement \cite[Theorem 6.1]{FishelVaziraniTwo}, and the inverses of $D^{mh-1}$ 
 are precisely the {\bf maximal} alcoves in the {\bf bounded} chambers of the $m$-Shi 
 arrangement \cite[Theorem 6.2]{FishelVaziraniTwo}. The right side of 
 Figure \ref{fig:simplex_shi} displays the inverses of the alcoves in $D^4$, 
 along with the $1$-Shi arrangement for $W=\sym_3$.  


\subsection{Lattice Paths}
\label{Lattice Paths}
Now consider the simplex $D^b$ corresponding 
to $W=\sym_a$, with $b$ coprime to $a$. 
(In Figure \ref{fig:simplex_shi} we have $a=3$ and $b=4$.) 
The positive alcoves in this simplex are counted by $\Cat(\sym_a,b)$ 
and they can be encoded by ``abacus diagrams" satisfying certain restrictions. 
It turns out that these are the same abacus diagrams that define the set of $(a,b)$-cores 
(i.e. integer partitions in which no cell has hook length equal to $a$ or $b$). 
Hence the number of such cores is the Catalan number $\Cat(\sym_a,b)$. 
This result was first proved by Anderson \cite{Anderson} using a different method. 
She gave a beautiful bijection from the set of $(a,b)$-cores to a 
collection of certain lattice paths, which we now define.

A {\sf rational Dyck path} is a path from $(0,0)$ to $(b,a)$ in the integer lattice $\ZZ^2$ using steps of the form $(1,0)$ and $(0,1)$ and staying above the diagonal $y = \frac{a}{b} x$.
(Because $a$ and $b$ are coprime, it will never {\bf touch} the diagonal.) 
More specifically, we call this an $x$-Dyck path or an $(a,b)$-Dyck path. 
For example, Figure \ref{fig:dyck_path} displays a $(5,8)$-Dyck path.
When $a$ and $b$ are clear from context, we will sometimes refer to
$(a,b)$-Dyck paths as simply `Dyck paths'.

In the proofs of Theorem~\ref{shellable} and Proposition~\ref{promotion-is-rotation} below,
we will use an alternative characterization of Dyck paths in terms of partitions.  A 
{\sf partition} $\lambda$ is a weakly decreasing sequence
$\lambda = (\lambda_1 \geq \dots \geq \lambda_k)$ of nonnegative integers.  
The number $k$ is called the number of {\sf parts} of the partition.
The 
{\sf Ferrers diagram} associated with a partition $\lambda$ consists of $\lambda_i$ left justified
boxes in row $i$ (this is the English notation).  

Given an $(a,b)$-Dyck path $D$, let $\lambda(D)$ be the partition with $a-1$ parts
whose Ferrers diagram 
is the northwest region traced out by $D$ inside the rectangle with corners $(0, 0)$ and
$(b, a)$.  For example, if $D$ is the $(5,8)$-Dyck path in Figure~\ref{fig:dyck_path},
we have that $\lambda(D) = (5,2,2,0)$.  It follows that a 
partition $\lambda = (\lambda_1 \geq \dots \geq \lambda_{a-1})$ with $a-1$ parts
comes from an $(a,b)$-Dyck path if and only if the parts of $\lambda$ satisfy
$\lambda_i \leq \mathrm{max}(\lfloor \frac{(a-i)b}{a} \rfloor, 0)$ for all $1 \leq i \leq a-1$.

For the proof of Proposition~\ref{promotion-is-rotation}, we will also think of an 
$(a,b)$-Dyck path $D$ as tracing out an order ideal  (i.e., a down-closed subset) $I = I(D)$
of the poset whose elements are the lattice squares inside the rectangle 
with corners $(0, 0)$ and $(b, a)$ and increasing directions north and west.  The
boxes in $I(D)$ form the complement of the Ferrers diagram of $\lambda(D)$.

\begin{figure}
\begin{center}
\includegraphics[scale=.7]{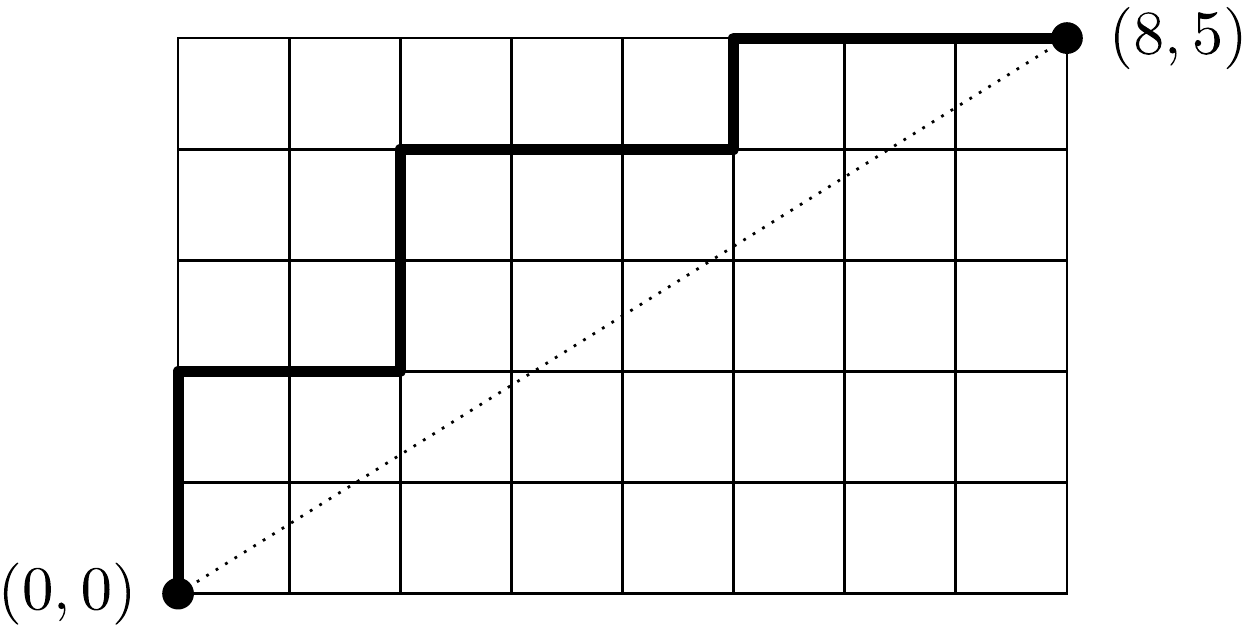}
\end{center}
\caption{This is a $(5,8)$-Dyck path.}
\label{fig:dyck_path}
\end{figure}

Note that the final step of an $(n,n+1)$-Dyck path must travel from $(n,n)$ to $(n,n+1)$. 
Upon removing this step we obtain a path from $(0,0)$ to $(n,n)$ that stays {\em weakly} 
above the line of slope $1$; that is, we obtain a classical Dyck path. 
  The following result generalizes
  the fact that there are $\Cat(n,n+1)$ classical Dyck paths and can be proven using the 
  Cycle Lemma
   of Dvorestky and Motzkin \cite{DvoretskyMotzkin}.  While this result is perhaps best
   attributed to `folklore', a proof was given by Bizley \cite{Bizley} in 1954 in the now-defunct Journal
   for the Institute of Actuaries.

\begin{thm} 
\label{actuary}
For $a\neq b$ coprime positive integers, the number of $(a,b)$-Dyck 
paths is the Catalan number $\Cat(a,b)=\frac{1}{a+b} \binom{a+b}{a,b}$.
\end{thm}

 Theorem~\ref{actuary}, as well as the following refinement, can be proven using the
Cycle Lemma.  We thank Nick Loehr for suggesting this argument.

\begin{thm}
\label{refined-dyck-proposition}
The number of $(a,b)$-Dyck paths with $i$ nontrivial vertical runs is the {\em Narayana number}
\begin{equation*}
\Nar(a,b;i):=\frac{1}{a}\binom{a}{i}\binom{b-1}{i-1},
\end{equation*}
and the number of $(a,b)$-Dyck paths with $r_j$ vertical runs of length $j$ is the {\em  Kreweras number}
\begin{equation*}
\Krew(a,b;\mathbf{r}):=\frac{1}{b}\binom{b}{r_0,r_1,\ldots,r_a}=\frac{(b-1)!}{r_0!r_1!\cdots r_a!}.
\end{equation*}
\end{thm}
Equivalently, the first formula counts the $(a,b)$-Dyck paths with $i-1$ ``valleys."  We include trivial vertical runs of ``length $0$'' in the second formula just to make it look nice. For example, the path in Figure \ref{fig:dyck_path} has $3$ nontrivial vertical runs (i.e. $2$ valleys) and $\mathbf{r}=(5,1,2,0,0,0)$. 
The rational Narayana numbers will appear below as the $h$-vector of the ``rational" associahedron.

\begin{proof} (Sketch.)
Fix a sequence ${\bf r} = (r_0, r_1, \dots, r_a)$ of nonnegative integers satisfying
$\sum r_j = b$ and $\sum j r_j = a$.  
By reading vertical run lengths from southwest to northeast,
we can think of an $(a,b)$-Dyck path with 
vertical run length sequence given by ${\bf r}$ as a length $b$ word in the letters
$x_0, x_1, \dots, x_a$ which contains $r_i$ occurrences of $x_i$ for all $i$.  For example,
the word corresponding to the $(5,8)$-Dyck path in Figure~\ref{fig:dyck_path} is
$x_2 x_0 x_2 x_0 x_0  x_1 x_0 x_0$.  

Let $W({\bf r})$ denote the set of all possible words in $x_0, x_1, \dots, x_a$ which contain
$x_j$ with multiplicity $r_j$.  Then $W({\bf r})$ is counted by the multinomial coefficient
\begin{equation*}
|W({\bf r})| = \binom{b}{r_0, r_1, \dots, r_a}.
\end{equation*}
One checks (using the coprimality of $a$ and $b$)
that each of the length $b$ words in $W({\bf r})$ has $b$ distinct cyclic 
conjugates and that exactly one of these conjugates corresponds to an 
$(a,b)$-Dyck path.  The desired Kreweras enumeration follows.

The Narayana enumeration also relies on a `cycle' type argument.  By reading
{\it horizontal} run lengths from northeast to southwest, we can think of an
$(a,b)$-Dyck path as a length $a$ word in the letters $y_0, y_1, \dots, y_b$.  
For example, the path in Figure~\ref{fig:dyck_path} corresponds to the 
word $y_3 y_3 y_0 y_2 y_0$.  For any such word coming from
a Dyck path, the number of nontrivial
vertical runs equals the number of letters $y_k$ with $k > 0$.

Let $W(i)$ be the set of all possible length $a$ words $y_{k_1} \cdots y_{k_a}$ in
$y_0, y_1, \dots, y_b$ such that $k_1 + \cdots + k_a = b$ and exactly 
$i$ elements of the sequence $(k_1, \dots, k_a)$ are nonzero.  We claim that 
\begin{equation*}
|W(i)| = {a \choose i}{b-1 \choose i-1}.
\end{equation*}
This is because there are ${a \choose i}$ ways to choose which $i$-element subset
of the terms in the sequence $(k_1, \dots, k_a)$ are nonzero.  Since the nonzero terms must
add up to $b$, they form a strict composition of $b$ with $i$ parts.
There are ${b-1 \choose i-1}$ of these.  One checks as before that every word in 
$W(i)$ has $a$ distinct cyclic conjugates, exactly one of which comes from an $(a,b)$-Dyck path.
The desired Narayana enumeration follows.
\end{proof}

\subsection{The Laser Construction}
\label{The Laser Construction}
Our generalizations of associahedra and noncrossing partitions to the rational case will be 
based on a topological decomposition of rational Dyck paths using `lasers'.  We
devote a subsection to this key construction.

Let $D$ be an $(a, b)$-Dyck paths with $a < b$ coprime and let $P = (i, j)$ be a lattice point
on $D$ other than the origin $(0, 0)$ which is the bottom of a north step of $D$.  
The {\sf laser fired from $P$} is the line segment $\ell(P)$ which has slope $\frac{a}{b}$,
southwest endpoint $P$, and northeast endpoint the lowest point higher than $P$
where the line $y - j = \frac{a}{b}(x-i)$ intersects the Dyck path $D$.  That is,
the segment $\ell(P)$ is obtained by firing a laser of slope
$\frac{a}{b}$ northeast from $P$, where we consider the Dyck path $D$ to be
`solid'. 
The lower left of Figure \ref{fig:DyckTriang} shows an example of a $(5, 8)$-Dyck path
$D$ with lasers $\ell(P)$
fired from every possible nonzero lattice point $P$ which is at the bottom of a north step in $D$.
In our constructions we will often be interested in firing lasers from only some of the possible 
lattice points in $D$.

By coprimality, the laser $\ell(P)$ does not intersect any lattice points other than $P$.  In particular,
the northeast endpoint of $\ell(P)$ intersects $D$ in the interior of an east step.  We will often
associate $\ell(P)$ with the pair of lattice points $\{P, Q\}$, where $Q$ is at the right end of this east step.
Also observe that the lasers $\ell(P)$ and $\ell(P')$ do not cross for $P \neq P'$ because they 
have the same slope.

\section{Rational Associahedra}
\label{Rational Associahedra}

\subsection{Simplicial Complexes}
\label{Simplicial Complexes}
We recall a collection of definitions related to simplicial
complexes.  A {\sf simplicial complex} $\Delta$ on a finite ground set $E$ is a collection
of subsets of $E$ such that if $S \in \Delta$ and $T \subseteq S$, then $T \in \Delta$.  
The elements
of $\Delta$ are called {\sf faces}, the maximal elements of $\Delta$ are 
called {\sf facets}, and $\Delta$ is called {\sf pure} if all of its facets have the same cardinality.
The {\sf dimension} of a face $S \in \Delta$ is $\dim(S) := |S| - 1$ and the {\sf dimension} of
$\Delta$ is the maximum dimension of a face in $\Delta$.  Observe that the `empty face'
$\emptyset$ has dimension $-1$.

A simplicial complex $\Delta$ on a ground set $E$ is called {\sf flag} if 
for any subset $F \subseteq E$, we have that $F$ is a face of $\Delta$ whenever
every two-element subset of $F$ is a face of $\Delta$.  Flag simplicial complexes
are therefore determined by their $1$-dimensional faces.

If $\Delta$ is a $d$-dimensional simplicial complex, the {\sf $f$-vector} of $\Delta$
is the integer sequence 
$f(\Delta) = (f_{-1}, f_0, \dots, f_d)$, where $f_{-1} = 1$ and
$f_i$ is the number of $i$-dimensional faces in $\Delta$ for $0 \leq i \leq d$.
The {\sf reduced Euler characteristic} $\chi(\Delta)$ is given by
$\chi(\Delta) := \sum_{i = -1}^d (-1)^i f_i$.  The {\sf $h$-vector} of
$\Delta$ is the sequence $h(\Delta) = (h_{-1}, h_0, \dots, h_d)$ defined 
by the following polynomial equation in $t$:
$\sum_{i=-1}^d f_i (t-1)^{d-i} = \sum_{k = -1}^d h_k t^{d-k}$. 
The sequences
$f(\Delta)$ and $h(\Delta)$ determine one another completely for 
any simplicial complex $\Delta$.

Shellability is a key property possessed by some pure simplicial complexes 
which determines the homotopy type and $h$-vector of the complex.
Let $\Delta$ be a pure $d$-dimensional simplicial complex.  A total order
$F_1 \prec \dots \prec F_r$ on the facets $F_1, \dots, F_r$ of $\Delta$ is 
called a {\sf shelling order} if for $2 \leq k \leq r$, the subcomplex of the simplex $F_k$ defined
by $C_k := (\bigcup_{i = 1}^{k-1} F_i) \cap F_k$ is a pure $(d-1)$-dimensional simplicial complex.
The complex $\Delta$ is called {\sf shellable} if there exists a shelling order on its facets;
it can be shown that any pure 
$d$-dimensional shellable simplicial complex is homotopy equivalent to a wedge of
spheres, all of dimension $d$.

For future use, we record the following sufficient (but not necessary) condition
for a total order on the facets of a complex to be a shelling order.
We also give a formula for the $h$-vector in this case.  
While this result is certainly known, the authors were unable to find a reference
 in the literature and include a proof for the sake of completeness.

\begin{lem}
\label{shelling-characterization}
Let $\Delta$ be a pure $d$-dimensional
simplicial complex and let $F_1 \prec \dots \prec F_r$ be a total
order on the facets of $\Delta$.   Suppose that for $1 \leq i \leq k$ there exists a unique minimal
face $M_k$ of the facet $F_k$ which is not contained in
the previous subcomplex $\bigcup_{i = 1}^{k-1} F_i$.  Then the order 
$\prec$ is a shelling order and 
the $i^{th}$ entry $h_i$ of the $h$-vector $h(\Delta)$ equals the number of minimal 
faces $M_k$ with $\dim(M_k) = i-1$.
\end{lem}

\begin{proof}
For $1 < k \leq r$, the intersection 
$(\bigcup_{i=1}^{k-1} F_i) \cap F_k$ is the simplicial complex with facets
$F_k - \{v\}$ for $v \in M_k$, and is therefore pure of dimension 
$d-1$.  We conclude that $\prec$ is a shelling order.  

By a result of Chan
\cite[Proposition 1.2.3]{Chan}
(see also McMullen \cite[p. 182]{McMullen}), the $h$-vector entry 
$h_i$ is the number of facets $F_k$ such that the addition of 
$F_k$ to the previous subcomplex $\bigcup_{j = 1}^{k-1} F_j$  adds
exactly $i$ new facets.  But the number of new facets added is just the 
dimension of $M_k$ plus one.
\end{proof}

\subsection{Construction, Basic Facts, and Conjectures}
\label{Construction, Basic Facts, and Conjectures}

For $n \geq 3$, let $\PP_n$ denote the regular $n$-gon.
Recall that the (dual of the) classical associahedron 
$\Ass(n,n+1)$\footnote{This notation will soon be justified.}
consists of 
all (noncrossing) dissections of $\PP_{n+2}$, ordered by inclusion.  The
diagonals of $\PP_{n+2}$ are therefore the vertices of $\Ass(n,n+1)$ and
the facets of $\Ass(n,n+1)$ are labeled by triangulations of
$\PP_{n+2}$.  Associahedra were introduced by Stasheff
\cite{Stasheff} in the context of 
nonassociative products arising in algebraic topology.  Since its introduction,
the associahedron has become one of the most well-studied complexes in 
geometric combinatorics, with connections to the permutohedron and 
exchange graphs of
cluster algebras.  

The classical associahedron has a Fuss analog defined as follows. 
Let $m \geq 1$ be a Fuss parameter.  The Fuss associahedron
$\Ass(n, mn+1)$ has as its facets the collection of all dissections of
$\PP_{mn+2}$ into $(m+2)$-gons.  Fuss associahedra arise in the 
study of the generalized cluster complexes of Fomin and Reading.

We define our further generalization $\Ass(a,b)$ of the classical associahedron
by describing its facets as follows.
Label the vertices of $\PP_{b+1}$ clockwise with $1, 2, \dots, b+1$.

Given any Dyck path $D$ and any lattice point $P$ 
which is the bottom of a north step in $D$, 
we associate a diagonal $e(P)$ in $\PP_{b+1}$ as follows.
Starting at the point $P$, consider the laser $\ell(P)$ fired from $P$.
As in Subsection~\ref{The Laser Construction}, we associate $P$ to the pair
of lattice points $\{P, Q\}$, where $Q$ is the right endpoint of the east step whose interior 
contains the northeast endpoint of $\ell(P)$.  We define $e(P)$ to be the diagonal $(i+1, j+1)$,
where $i$ is the $x$-coordinate of $P$ and $j$ is the $x$-coordinate of $Q$.
We let $F(D)$ be the set of  possible `laser diagonals' corresponding to $D$:
\begin{equation}
F(D) := \{ e(P) \,:\, \mbox{$P$ is the bottom of a north step in $D$} \}.
\end{equation}
The right of Figure~\ref{fig:DyckTriang} shows the collection $F(D)$ of diagonals corresponding 
to the given Dyck path $D$ on $\PP_9$.
Observe that if we did not assume that $a < b$, the `diagonals' described by the 
set $F(D)$ might join adjacent vertices of $\PP_{b+1}$.
It is topologically clear that the collection $F(D)$ of diagonals in $\PP_{b+1}$ is noncrossing for
any Dyck path $D$.  The sets $F(D)$ form the facets of our simplicial complex.

\begin{figure}
\begin{center}
\includegraphics[scale=.9]{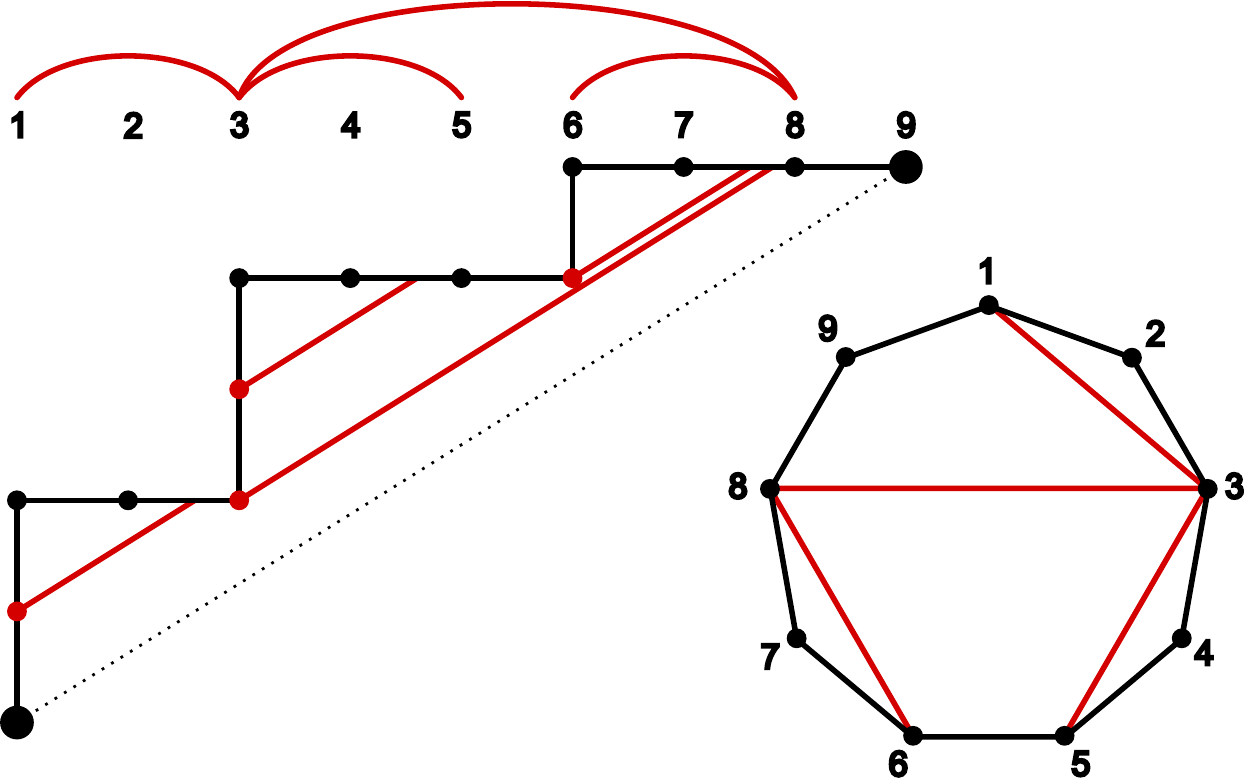}
\end{center}
\caption{A (5,8)-Dyck path and the corresponding dissection of $\PP_9$.}
\label{fig:DyckTriang}
\end{figure}

\begin{defn}
Let $\Ass(a,b)$ be the simplicial complex with facets
\begin{equation}
\{ F(D) \,:\, \text{$D$ is an $(a,b)$-Dyck path} \}.
\end{equation}
\end{defn}

Figure~\ref{fig:Alexander} shows the complex $\Ass(3,5)$ in red and the complex
$\Ass(2,5)$ in blue.  These complexes are embedded inside the larger associahedron
$\Ass(4,5)$ of dissections of $\PP_6$.

It is natural to ask which diagonals of $\PP_{b+1}$ appear as vertices in
$\Ass(a,b)$.  The proof of the following proposition follows from the geometry 
of lines of slope $\frac{a}{b}$ and is left to the reader.

\begin{prop}
\label{admissible}
Define a subset $S(a,b)$ of $[b-1]$ by 
$S(a,b) = \{ \lfloor \frac{ib}{a} \rfloor \,:\, 1 \leq i < a \}$, where
$\lfloor s \rfloor$ is the greatest integer $\leq s$.  A diagonal of $\PP_{b+1}$ which
separates $i$ vertices from $b-i-1$ vertices appears as a vertex in the complex
$\Ass(a,b)$ if and only if $i \in S(a,b)$.
\end{prop}

A diagonal of $\PP_{b+1}$ which appears as a vertex of $\Ass(a,b)$ will be called
{\sf $(a,b)$-admissible}.
The following basic facts about $\Ass(a,b)$ can be proven directly from its definition.

\begin{prop}
\label{basic-facts}
\begin{enumerate}
\item
The simplicial complex $\Ass(a,b)$ is pure and
has dimension $a-2$.

\item
The number of facets in $\Ass(a,b)$ is $\Cat(a,b)$.

\end{enumerate}
\end{prop}

\begin{proof}
Part 1 follows from the fact that an $(a,b)$-Dyck path contains $a$ north steps.
For Part 2, observe that if $D$ and $D'$ are distinct Dyck paths, the multisets of
$x$-coordinates of the bottoms of the north steps of $D$ and $D'$ are distinct.  
In particular, this means that $F(D)$ and $F(D')$ are distinct sets of diagonals in 
$\PP_{b+1}$.  Part 2 follows from the fact that there are $\Cat(a,b)$ Dyck paths.  
\end{proof}

In the case $b \equiv 1$ (mod $a$), we have the following more 
widely used description of the complex $\Ass(a,b)$.

\begin{prop}
\label{fuss-reduction}
Assume that $b = ma + 1$.  Then $\Ass(a,b)$ is the simplicial complex 
whose faces are collections of mutually noncrossing $(a,b)$-admissible 
diagonals in $\PP_{b+1}$.  In particular, the complex
$\Ass(a,b)$ is flag and carries an action of the cyclic group
$\ZZ_{b+1}$ given by rotation.
\end{prop}

\begin{proof}
Let $\Delta$ be the complex so described.  
Certainly $\Ass(a,b) \subseteq \Delta$.
The facets of $\Delta$ are precisely
the dissections of $\PP_{b+1} = \PP_{ma+2}$ into $(m+2)$-gons.  It is  well known that 
the number of such dissections is the Fuss-Catalan number
$\Cat(a,b) = \Cat(a, ma+1)$.  By Part 2 of Proposition~\ref{basic-facts},
this is also the number of facets of $\Ass(a,b)$.  Since complexes 
$\Ass(a,b)$ and $\Delta$ have the same collection of facets, we 
conclude that $\Ass(a,b) = \Delta$.
\end{proof}

Proposition~\ref{fuss-reduction} is false at the full rational level of 
generality.
Indeed, when $(a, b) = (3, 5)$, the diagonals $(1, 5)$ and $(3, 5)$ of $\PP_6$ are 
$(3, 5)$-admissible and mutually noncrossing.  
However, 
the set $\{ (1, 5), (3, 5) \}$ is not a face of $\Ass(3, 5)$.  
A glance at Figure~\ref{fig:Alexander} shows that the red complex
$\Ass(3, 5)$ is not closed under rotation of $\PP_6$.

In spite of the last paragraph, we conjecture that $\Ass(a, b)$ carries a rotation action 
`up to homotopy'.  More precisely, we make the following definition.

\begin{defn}
Let $\Asshat(a,b)$ denote the simplicial complex whose faces are collections of
mutually noncrossing $a,b$-admissible diagonals in $\PP_{b+1}$.
\end{defn}

It is clear that $\Asshat(a, b)$ carries a rotation action and that
$\Ass(a,b)$ is a flag subcomplex of $\Asshat(a,b)$.  
When $b \equiv 1$ (mod $a$) the complexes 
$\Ass(a,b)$ and $\Asshat(a,b)$ coincide.

Before stating our conjecture, we recall what it means for a complex to collapse
onto a subcomplex;  this is a 
combinatorial deformation 
retraction.  Let $\Delta$ be a simplicial complex, $F \in \Delta$ be a facet, and suppose
$F' \subset F$ satisfies $|F'| = |F| - 1$. If $F'$ is not contained in any facet of $\Delta$ besides $F$,
we can perform an {\sf elementary collapse} by replacing 
$\Delta$ with $\Delta - \{F, F'\}$.  A simplicial complex is said to {\sf collapse} onto a subcomplex
if the subcomplex can be obtained by a sequence of elementary collapses.

\begin{conj}
\label{homotopy-equivalent}
The complexes $\Ass(a,b)$ and 
$\Asshat(a,b)$ are homotopy equivalent.  In fact, the complex 
$\Asshat(a,b)$ collapses onto the subcomplex $\Ass(a,b)$.
\end{conj}

Figure~\ref{fig:Alexander} displays $\Ass(2,5)$ (shown in blue) and
$\Ass(3,5)$ (shown in red) as subcomplexes of the sphere $\Ass(4,5)$.  
The complex $\Asshat(2,5)$ coincides with $\Ass(2,5)$ and the complex 
$\Asshat(3,5)$ is obtained from the complex
$\Ass(3,5)$ by adding  the front and back triangles
to the red complex.  Observe that 
$\Ass(3,5)$ can be obtained by performing two elementary collapses
on $\Asshat(3,5)$.

Conjecture~\ref{homotopy-equivalent} would also have implications regarding 
Alexander duality.
Recall that two topological subspaces $X$ any $Y$  of a fixed sphere $S$ are said to be
{\sf Alexander dual} to one another if $Y$ is homotopy equivalent to the complement of $X$
in $S$.
With $b > 1$ fixed, we have that $a$ and $b$ are coprime for 
$1 \leq a < b$ if and only if $b-a$ and $b$ are coprime.  
Both of the complexes $\Ass(a,b)$ and $\Ass(a-b,b)$ sit within the classical associahedron
$\Ass(b-1,b)$.
The proof of
Conjecture~\ref{homotopy-equivalent} would imply that 
$\Ass(a,b)$ and $\Ass(a-b,b)$ are Alexander dual.

\begin{figure}
\begin{center}
\includegraphics[scale=.6]{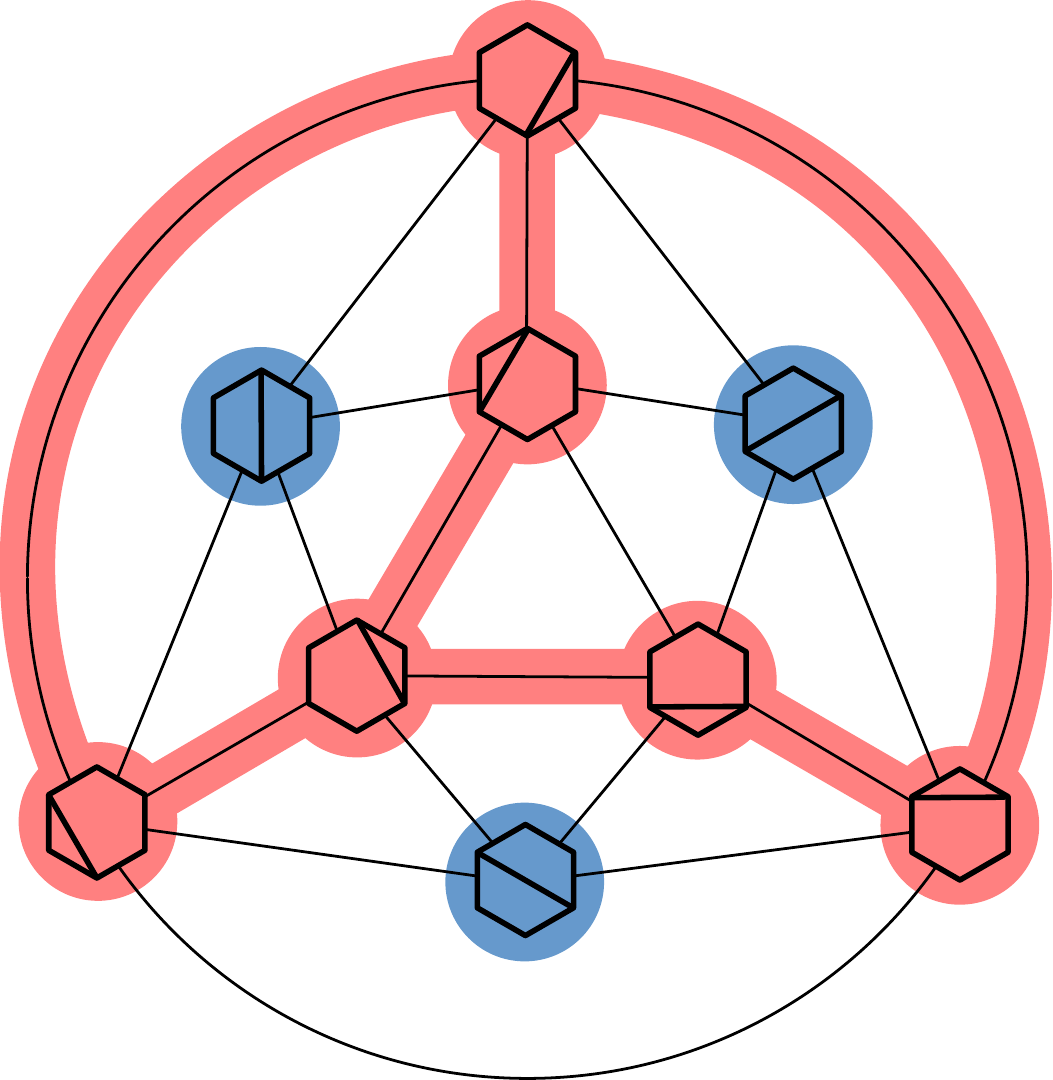}
\end{center}
\caption{$\Ass(2,5)$ and $\Ass(3,5)$ are Alexander dual within $\Ass(4,5)$.}
\label{fig:Alexander}
\end{figure}

\begin{prop}
\label{alexander-duality}
Let $a < b$ be coprime
for $b > 1$.  
The subcomplexes $\Asshat(a,b)$ and $\Asshat(b-a,b)$ are Alexander dual within
the sphere $\Ass(b-1,b)$.  If Conjecture~\ref{homotopy-equivalent} is true, then
the subcomplexes $\Ass(a,b)$ and $\Ass(b-a,b)$ are also Alexander dual
within  $\Ass(b-1,b)$.
\end{prop}

\begin{proof}
It is routine to check that any diagonal of $\PP_{b+1}$ is either 
$(a,b)$-admissible or $(b-a,b)$-admissible, but not both.  This means that the 
vertex sets of $\Asshat(a,b)$ and $\Asshat(b-a,b)$ partition the vertex set of 
the simplicial sphere $\Asshat(b-1,b)$.  By definition, the faces of 
$\Asshat(a,b)$ and $\Asshat(b-a,b)$ are precisely the faces of $\Ass(b-1,b)$ whose
vertex sets are contained in $\Asshat(a,b)$ and $\Asshat(b-a,b)$, respectively.
It follows that the complement of $\Asshat(a,b)$ inside $\Ass(b-1,b)$ deformation retracts
onto $\Asshat(b-a,b)$.  This proves the first statement.  The second statement is clear.
\end{proof}

\subsection{Shellability and $f$- and $h$-vectors}
\label{Shellability and $f$- and $h$-vectors}

We will prove that the simplicial complex $\Ass(a, b)$ is shellable by giving
an explicit shelling order on its facets.  This shelling order will be induced by lexicographic
order on the partitions whose Ferrers diagrams lie to the northwest of $(a,b)$-Dyck paths.

\begin{thm}
\label{shellable}
The simplicial complex $\Ass(a,b)$ is shellable, hence homotopy equivalent to a wedge 
of spheres.  Moreover, there is a total order $D_1 \prec D_2 \prec \dots \prec D_{\Cat(a,b)}$
on the set of $(a,b)$-Dyck paths which induces a shelling order on the facets of 
$\Ass(a,b)$ such that the dimension of the minimal face added upon addition of the facet
$F(D_i)$ equals the number of nonempty vertical runs in $D_i$, less one.
\end{thm}

\begin{proof}
We will find it convenient to identify the facets of $\Ass(a,b)$ with both Dyck 
paths and partitions.  For this proof we will use the {\sf lexicographical} order 
on partitions with $a-1$ parts defined by 
$\lambda \prec \mu$ if there exists $1 \leq i \leq a-1$ such that
$\lambda_i < \mu_i$ and $\lambda_j = \mu_j$ for all $1 \leq j < i$.

Let $\lambda^{(1)} \prec \dots \prec \lambda^{(\Cat(a,b))}$ be the restriction of lexicographic order to
set of partitions with $a-1$ parts
which satisfy $\lambda_i \leq \mathrm{max}(\lfloor \frac{(a-i)b}{a} \rfloor, 0)$ for all $i$, that is,
those partitions coming from $(a,b)$-Dyck paths.
In particular, we have that $\lambda^{(1)}$ is the empty partition and 
$\lambda^{(\Cat(a,b))}_i =  \mathrm{max}(\lfloor \frac{(a-i)b}{a} \rfloor, 0)$.  
The total order $\prec$ induces a
total order 
$D_1 \prec \dots \prec D_{\Cat(a,b)}$
on $(a,b)$-Dyck paths and a total order 
$F(D_1) \prec \dots \prec F(D_{\Cat(a,b)})$
on the facets of $\Ass(a,b)$.  

In the case $(a, b) = (3, 5)$, our order on partitions is
\begin{equation*}
(0,0) \prec (1,0) \prec (1,1) \prec (2,0) \prec (2,1) \prec (3,0) \prec (3,1).
\end{equation*}
The corresponding order on facets of $\Ass(3,5)$ (written as diagonal sets in $\PP_6$) is
\begin{align*}
&\{ (1,3), (1,5) \} \prec \{ (2,4), (1,5) \} \prec
\{ (2,4), (2,6) \} \prec \\ &\{ (1,3), (3,5) \} \prec 
\{ (2,6), (3,5) \} \prec \{ (1,3), (4,6) \} \prec
\{ (2,4), (4,6) \}.
\end{align*}

We will prove that $\prec$ is a shelling order on the facets of $\Ass(a,b)$
and that the minimal added faces corresponding to $\prec$ have the required dimensions.  In fact,
we will be able to describe these minimal added faces explicitly.  Given any Dyck path $D$, recall that
the corresponding facet $F(D)$ in $\Ass(a,b)$ is given by
$F(D) = \{ e(P) \,:\, P$ is the bottom of a north step in $D \}$.  
We define the {\sf valley face} $V(D)$ to be the subset of $F(D)$ given by 
$V(D) := \{ e(P) \,:\, P$ is a valley in $D \}$.

In the case $(a, b) = (3, 5)$, the valley faces $V(D)$ written in the order $\prec$ are
\begin{equation*}
\emptyset \prec \{ (2,4) \} \prec
\{ (2,6) \} \prec \{ (3,5) \} \prec 
\{ (2,6), (3,5) \} \prec \{ (4,6) \} \prec
\{ (2,4), (4,6) \}.
\end{equation*}
The reader is invited to check that $\prec$ is a shelling order on the facets of 
$\Ass(3,5)$ and that the valley face is the minimal face added at each stage.  
Keeping track of the sizes of the valley faces, this recovers the fact that the $h$-vector of 
$\Ass(3,5)$ is $(1,4,2)$.  We claim that this is a general phenomenon.

{\bf Claim:} {\it For $1 \leq k \leq \Cat(a,b)$, the valley face $V(D_k)$ is the unique minimal face of
$F(D_k)$ which is not contained in $\bigcup_{i = 1}^{k-1} F(D_i)$.}

Let $1 \leq k \leq \Cat(a,b)$.
The proof of our claim comes in two parts:  we first show that $V(D_k)$ is not contained in
$\bigcup_{i = 1}^{k-1} F(D_i)$ and then show that if $T_k$ is any face of $F(D_k)$ which does not 
contain $V(D_k)$, then $T_k$ is contained in $\bigcup_{i=1}^{k-1} F(D_i)$.  We break this up into 
two lemmas.

\begin{lem}
\label{first-lemma}
Let $1 \leq k \leq \Cat(a,b)$.  The valley face $V(D_k)$ is not contained in
$\bigcup_{i=1}^{k-1} F(D_i)$.
\end{lem}

\begin{proof}

When $k = 1$, the Dyck path $D_1$ has $\lambda(D_1)$ equal to the empty partition, the valley
face $V(D_1)$ is the empty face, and $V(D_1)$ is not contained in the void complex
$\bigcup_{i=1}^{k-1} F(D_i)$.  Assume therefore that $2 \leq k \leq r$ and suppose there exists
$1 \leq i \leq k$ such that $V(D_k)$ is contained in $F(D_i)$.  We will prove that 
$\lambda(D_i) = \lambda(D_k)$.  Indeed, let $P_1 = (x_1, y_1), \dots, P_s = (x_s, y_s)$ be the 
set of valleys of $D_k$ from right to left, so that 
$x_1 > \dots > x_s > 0$ and $y_1 > \dots > y_s > 0$.  
(Since $k > 1$, the Dyck path $D_k$ has at least one
valley, so $s \geq 1$.)  This implies that 
$\lambda(D_k) = (x_1^{a - y_1}, x_2^{y_2 - y_1}, \dots, x_s^{y_{s-1} - y_s})$, where
exponents denote repeated parts.  Since $V(D_k) \subseteq F(D_i)$, we have that
$e(P_1) \in F(D_i)$.  This forces $x_1$ to appear as a part of the partition
$\lambda(D_i)$.  Since $\frac{b}{a} > 1$, by geometric considerations involving the slope of the
laser $\ell(P_1)$ defining $e(P_1)$
the minimum multiplicity with which $x_1$ could occur
as a part of $\lambda(D_i)$ is $a - y_1$.  The fact that $\lambda(D_i) \preceq \lambda(D_k)$
forces $x_1$ to appear with multiplicity {\it exactly} equal to $a - y_1$ in $\lambda(D_i)$, and any
parts $> x_1$ to appear with multiplicity zero in $\lambda(D_i)$.  
In other words, the partition $\lambda(D_i)$ has the form $(x_1^{a - y_1}, \dots)$, where the parts
after $x_1^{a - y_1}$ are all $< x_1$.
We now focus on the valley $P_2$ of $D_k$.  Since $V(D_k) \subseteq F(D_i)$, we have that 
$e(P_2) \in F(D_i)$.  This forces $x_2$ to appear as a part of the partition
$\lambda(D_i)$.  Since we already know that $\lambda(D_i)$ has the form
$(x_1^{a - y_1}, \dots )$, where the parts after $x_1^{a - y_1}$ are all $< x_1$, geometric considerations
involving the slope of the laser $\ell(P_2)$ defining $e(P_2)$ together with the fact that
$\lambda(D_i) \preceq \lambda(D_k)$ force $\lambda(D_i)$ to be of the form
$(x_1^{a - y_1}, x_2^{y_1 - y_2}, \dots)$, where the parts occurring after
$x_1^{a-y_1}, x_2^{y_1 - y_2}$ are all $< x_2$.  Iterating this process with the valleys 
$P_3, P_4, \dots, P_s$, we eventually get that 
$\lambda(D_i)$ has the form $(x_1^{a - y_1}, x_2^{y_1 - y_2}, \dots, x_s^{y_{s-1} - y_s}, \dots)$,
where the parts occurring in the ellipses are all $< x_s$.  But the fact that
$\lambda(D_i) \preceq \lambda(D_k)$ forces $\lambda(D_i) = \lambda(D_k)$.
The completes the proof that $V(D_k)$ does not appear in $\bigcup_{i = 1}^{k-1} F(D_i)$.
\end{proof}

\begin{lem}
\label{second-lemma}
Let $1 \leq k \leq r$ and let $T_k$ be any face of $F(D_k)$ which does not contain
$V(D_k)$.  Then $T_k$ is contained in $\bigcup_{i=1}^{k-1} F(D_i)$.
\end{lem}

\begin{proof}
Without loss of generality we may assume that $T_k$ is maximal among the subsets
of $F(D_k)$ which do not contain $V(D_k)$.
This means that there 
exists a valley $P$ of the Dyck path $D_i$ such that 

\begin{equation}
T_k = \{ e(Q) \,:\, \mbox{$Q$ is  the bottom of a north step in $D_k$ and $Q \neq P$} \}.
\end{equation}

Since $D_1$ does not have any valleys, we have that $k > 1$.
We will show that $T_k$ is contained in $\bigcup_{i=1}^{k-1} F(D_i)$.

We can factor 
$D_k$ into north and east runs as 
$D_k = N^{i_1} E^{j_1} \cdots N^{i_n} E^{j_n}$,
where each of the north and east runs are nonempty.  Let $1 \leq r < n$ be such that 
$P$ is at the end of the east run $E^{j_r}$.  

We will produce a new Dyck path $D_k'$ such that $D_k'  \prec D_k$ and
$T_k \subset F(D_k')$.  Roughly speaking,
the path $D_k'$ will be built from the path $D_k$ by raising certain
east runs of $D_k$ by one unit.  
More formally, let $s$ be the maximal number $\leq r$ such that there exists 
a laser emanating from a lattice point on the north run
$N^{i_s}$ of $D_k$ which intersects $D_k$ in a point to the east of $P$.  (If no such laser
exists, set $s = 0$.)  We define our new path $D_k'$ in terms of north and east runs by

\begin{equation}
D_k' = N^{i_1 } E^{j_1}  \cdots 
N^{i_r + 1} E^{j_r} N^{i_{r+1}-1} E^{j_{r+1}} \cdots N^{i_n} E^{j_n},
\end{equation}
if $s= 0$, or
\begin{equation}
D_k' = N^{i_1} E^{j_1} \cdots N^{i_s + 1} E^{j_s} \cdots
N^{i_r - 1} E^{j_r} N^{i_{r+1}} E^{j_{r+1}} \cdots N^{i_n} E^{j_n},
\end{equation}
if $1 \leq s \leq r$.
In other words,
if $1 \leq s \leq r-1$, 
$D_k'$ is formed from $D_k$ by stretching the vertical run $N^{i_s}$ by one unit
and by shrinking $N^{i_r}$ by one unit. If $s = 0$, $D_k'$ is formed from
$D_k$ by stretching $N^{i_r}$ by one unit and shrinking $N^{i_{r+1}}$ by one unit.  In either case,
the point $P$ does not appear in the lattice path $D_k'$, the paths $D_k$ and $D_k'$ 
agree to the northeast of $P$, and we have that
$D_k' \prec D_k$.

\begin{figure}
\begin{center}
\includegraphics[scale=1.2]{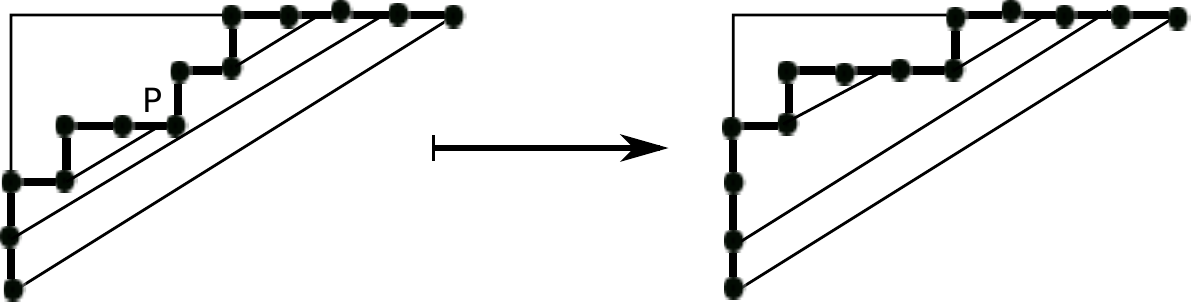}
\end{center}
\caption{The construction $D_k \mapsto D_k'$.}
\label{fig:sequalsone}
\end{figure}

Figure~\ref{fig:sequalsone} shows an example of the construction $D_k \mapsto D_k'$ when
$(a, b) = (5, 8)$.  The Dyck path $D_k$ 
is shown on the left and
factors as $N^2 E^1 N^1 E^2 N^1 E^1 N^1 E^4$, so that
$n = 4$, $(i_1, i_2, i_3, i_4) = (2, 1, 1, 1)$, and $(j_1, j_2, j_3, j_4) = (1, 2, 1, 4)$.
For the given valley  $P = (3,3)$, we have that $r = 2$ and $s = 1$.  To construct $D_k'$ from $D_k$, we
extend the north run $N^{i_1}$ in $D_k$ by one unit and shrink the north run 
$N^{i_3}$ in $D_k$ by one unit.    The resulting path $D_k'$ is shown on the right
of Figure~\ref{fig:sequalsone} and factors as
$N^3 E^1 N^1 E^3 N^1 E^4$. 
Observe that $\lambda(D_k) = (4,3,1)$ and $\lambda(D_k') = (4,1)$, so that we have
the lexicographic comparison $D_k' \prec D_k$.

We claim that $T_k \subset F(D_k')$.  
For the example in Figure~\ref{fig:sequalsone}, the lasers which correspond to the elements in
$T_k$ are shown on the Dyck path $D_k$ on the left; observe that a laser is fired from every 
possible vertex other than the valley $P$.  On the right, we have drawn
a subset of the lasers corresponding to the elements in $F(D_k')$ such that this subset
coincides with $T_k$.  Observe that we have fired a laser from every vertex which is the bottom of 
a north step in $D_k'$ except for a single vertex in the `stretched' north run.

To see that $T_k \subset F(D_k')$ in general, consider a lattice point $Q$ which is the 
bottom of a north step in $D_k$ such that $Q \neq P$.  We will show that $e(Q)$ appears as 
a vertex in the facet $F(D_k')$.  This breaks up into several cases depending on the position of $Q$.

If $Q$ is to the northeast of $P$, i.e., $Q$ is contained in a north run
$N^{i_m}$ for $m \geq r$, then $e(Q)$ is contained in the facet $F(D_k')$ because the paths $D_k$
and $D_k'$ agree to the northeast of $P$.

For example, in Figure~\ref{fig:sequalsone}, the vertex $Q = (4,4)$ lies to the northeast of $P$ 
on $D_k$ and its position (as well as the laser emanating from it) remains unchanged
in $D_k'$.

If $Q$ appears in a north run $N^{i_m}$ of $D_k$ for $s < m \leq r$ and $s > 0$, all of the 
lasers emanating from lattice points
in the north run $N^{i_m}$ intersect $D_k$ to the west of $P$.  
Since the portion of $D_k'$ between $Q$ and $P$ is just the corresponding portion of 
$D_k$ shifted north one unit, it follows that if $Q = Q' + (0, 1)$ is $Q$ shifted up one unit,
then $Q'$ is the bottom of a north step in $D_k'$ and the diagonal $e(Q)$ coming from $D_k$
equals the diagonal $e(Q')$ coming from $D_k'$.

For example, in Figure~\ref{fig:sequalsone}, the vertex $Q = (1,2)$  on $D_k$ satisfies the conditions
of the preceding paragraph.  This vertex and its laser are translated up one unit in the 
transformation $D_k \mapsto D_k'$.  This has no effect on the horizontal endpoint of the laser,
and hence no effect on the corresponding diagonal in $\PP_9$.

If $Q$ appears in the north run  $N^{i_s}$ of $D_k$ and $s > 0$, then the laser 
$\ell(Q)$ may intersect $D_k$ either to the east or west of $P$.
By construction, the path $D_k'$ is obtained from the path $D_k$ by stretching the 
vertical run $N^{i_s}$ by one unit.  If $\ell(Q)$ intersects $D_k$ to the east of $P$,
then we have that the vertex $e(Q)$ coming from $D_k$ equals the vertex
$e(Q)$ coming from $D_k'$.  On the other hand, if $\ell(Q)$ 
intersects $D_k$ to the west of $P$, then $Q' = Q + (0, 1)$ is the bottom of a north step in
$D_k'$, and the vertex $e(Q)$ coming from $D_k$ equals the vertex
$e(Q')$ coming from $D_k'$.

For example, in Figure~\ref{fig:sequalsone}, the point $Q = (0,1)$ on $D_k$ satisfies the 
conditions of the preceding paragraph.  Since the laser emanating from $Q$ hits $D_k$ to 
the east of $P$, we see that the laser in $D_k'$ emanating from $Q$ has endpoints with the 
same $x$-coordinates.

If $Q$ appears in a north run $N^{i_m}$ of $D_k$ for $m < s$ and $s > 0$, then
$\ell(Q)$ either intersects $D_k$ at a point to the 
east of $P$ or to the west of the east run $E^{i_s}$.  However, the lattice paths 
$D_k$ and $D_k'$ agree in these two regions.  It follows that $Q$ remains the bottom of a north
step in $D_k'$ and that the vertex $e(Q)$ coming from $D_k$ equals the vertex 
$e(Q)$ coming from $D_k'$.

Finally, if $Q$ appears in a north run $N^{i_m}$ of $D_k$ for $0 \leq m \leq r$ and $s = 0$,
then all of the lasers emanating from lattice points in the north
run $N^{i_m}$ intersect $D_k$ to the east of $P$.    By the construction of $D_k'$, the point
$Q$ also appears as the bottom of a north step in the path $D_k'$.  Since
$D_k$ and $D_k'$ agree to the east of $P$ and $D_k'$ is obtained from $D_k$ by shifting
a east run of $D_k$ up one unit, we have that the vertex $e(Q)$ coming from $D_k$ 
equals the vertex $e(Q)$ coming from $D_k'$.

We conclude that $T_k \subset F(D_k')$ and $D_k' \prec D_k$.
\end{proof}

Lemmas~\ref{first-lemma} and \ref{second-lemma} complete the proof of our claim that 
the valley face
$V(D_k)$ is indeed 
the unique minimal face in $F(D_k)$ which is not contained in
$\bigcup_{i=1}^{k-1} F(D_i)$.  This completes the proof of the Theorem.
\end{proof}

As a corollary to the above result, we get product formulas for the $f$- and $h$-vectors
of $\Ass(a,b)$, as well as its reduced Euler characteristic.  
Define the {\sf rational Kirkman
numbers}
by
\begin{equation}
\Kirk(a, b; i) := \frac{1}{a} {a \choose i}{b+i-1 \choose i-1}.
\end{equation}

\begin{cor}
\label{f-and-h}
Let $(f_{-1}, f_0, \dots, f_{a-2})$ and
$(h_{-1}, h_0, \dots, h_{a-2})$ be the $f$- and $h$-vectors of 
$\Ass(a,b)$.  For $1 \leq i \leq a$ we have that 
$f_{i-2} = \Kirk(a, b;i)$ and $h_{i-2} = \Nar(a,b;i)$.  The reduced
Euler characteristic of $\Ass(a,b)$ is $(-1)^{a+1}$
times the derived Catalan number $\Cat'(a,b)$.
\end{cor}

Conjecture~\ref{homotopy-equivalent} and Proposition~\ref{alexander-duality}
assert that 
the symmetry $(a < b) \leftrightarrow (b-a < b)$ on pairs of coprime positive integers
shows up in rational associahedra as an instance of Alexander duality.
Corollary~\ref{f-and-h} shows that the categorification
$\Cat(x) \mapsto \Cat'(x)$ of the Euclidean algorithm presented
in Section~\ref{Rational Catalan Numbers} sends the number of facets of $\Ass(a,b)$
to the reduced Euler characteristic of $\Ass(a,b)$.
This `categorifies' the number theoretic properties of rational Catalan numbers 
to the context of associahedra.

\begin{proof}
For this proof we will need the standard extension
${n \choose k} := \frac{n(n-1) \cdots (n-k+1)}{k!}$ of the binomial coefficient
to any $n \in \ZZ$ and the Vandermonde convolution
$\sum_{i = 0}^{k} {n \choose i}{m \choose k-i} = {n + m \choose k}$ which holds
for any $m, n, k \in \ZZ$ with $k \geq 0$.

By Theorem~\ref{shellable} and Lemma~\ref{shelling-characterization},  we have that
$h_{i-2}$ equals the number of $(a,b)$-Dyck paths which have exactly $i$ vertical runs.
By Theorem~\ref{refined-dyck-proposition}, this equals the Narayana number
$\Nar(a, b; i)$.

To prove the statement about the $f$-vector, one must check that 
\begin{equation}
\label{f-and-h-relation}
\sum_{i = -1}^{a-2} \Kirk(a, b; i+2) (t-1)^{a-i-2} = \sum_{k = -1}^{a-2} \Nar(a, b; k+2) t^{a-2-k}.
\end{equation}
Applying the transformation 
$t \mapsto t+1$, expanding in $t$,
and equating the coefficients of
$t^{a-i-2}$ on both sides of Equation~\ref{f-and-h-relation} 
yields the equivalent collection of binomial relations
\begin{equation}
\label{coefficient-identities}
\frac{1}{a} {a \choose i}{b+i-1 \choose i-1} =
\sum_{k=1}^i \frac{1}{a} {a \choose k}{b-1 \choose k-1}{a-k \choose a-i}
\end{equation}
for $1 \leq i \leq a$.  To prove Equation~\ref{coefficient-identities}, one uses the following chain
of equalities:
\begin{align*}
\sum_{k=1}^i \frac{1}{a} {a \choose k}{b-1 \choose k-1}{a-k \choose a-i} &=
\frac{1}{a}  \sum_{k=1}^i \frac{a! (a-k)!}{k!(a-k)!(a-i)!(i-k)!} {b-1 \choose k-1} \\
&= \frac{1}{a} \sum_{k = 1}^i \frac{a!}{(a-i)! i!} \frac{i!}{k! (i-k)!}  {b-1 \choose k-1}   \\
&= \frac{1}{a} \sum_{k = 1}^i {a \choose i} {i \choose k} {b-1 \choose k-1}  \\
&= \frac{1}{a} {a \choose i} \sum_{k = 1}^i {i \choose i-k} {b-1 \choose k-1} \\
&= \frac{1}{a} {a \choose i}{b+1-1 \choose i-1},
\end{align*}
where the final equality uses the Vandermonde convolution.

The statement about the Euler characteristic reduces to proving that
\begin{equation}
\label{euler}
\sum_{i = -1}^{a-2} (-1)^{i+1} \Kirk(a, b; i+2) = (-1)^{a+1} \Cat'(a,b).
\end{equation}
Recalling that $\Cat'(a,b) = \frac{1}{b} {b \choose a}$ and
$\Kirk(a, b; i+2) = \frac{1}{a} {a \choose i} {b + i - 1 \choose i - 1}$, 
we have the following
chain of equalities:
\begin{align*}
\label{euler-expanded}
\sum_{i=1}^a \frac{(-1)^{i+1}}{a} {a \choose i} {b+i-1 \choose i-1} &=
\sum_{i=1}^a \frac{(-1)^{i+1}}{b} {a-1 \choose i-1} {b+i-1 \choose i} \\
&= \sum_{i=1}^a \frac{(-1)^{2i+1}}{b} {a-1 \choose i-1}{-b \choose i} \\
&= - \frac{1}{b} \sum_{i=1}^a {a-1 \choose a-i}{-b \choose i} \\
&= - \frac{1}{b} {-b+a-1 \choose a} \\
&= (-1)^{a+1} \frac{1}{b} {b \choose a}.
\end{align*}
The fourth equality uses the Vandermonde convolution.  
\end{proof}

\section{Rational Noncrossing ``Matchings"}
\label{Rational Noncrossing Matchings}

\subsection{Construction, Basic Properties}
\label{Construction, Basic Properties}
Recall that 
a (perfect) matching $\mu$ on $[2n]$ is said to be {\sf noncrossing} if there do not exist indices
$1 \leq a < b < c < d \leq 2n$ such that $a \sim c$ and $b \sim d$ in $\mu$.  
There exist bijections between the set of noncrossing matchings on $[2n]$, the set of
standard Young tableaux of shape $2 \times n$, and noncrossing partitions on $[n]$
which send rotation on noncrossing matchings to promotion on tableaux to
Kreweras complementation on noncrossing partitions \cite{White}.  
We define a rational extension of noncrossing matchings;  rational analogs of
standard tableaux and noncrossing partitions are less well understood.

As with the case of the rational associahedron $\Ass(a,b)$, we will use $(a,b)$-Dyck paths to define 
rational analogs of noncrossing matchings.  We begin by defining
the rational analog of noncrossing matchings.  These will no longer be matchings in general,
so we call them {\sf homogeneous $(a,b)$-noncrossing partitions} (where we omit reference to
$(a,b)$ when it is clear from context).  

Let $D$ be an $(a,b)$-Dyck path.  We define a noncrossing set partition $\mu(D)$ of
$[a+b-1]$ as follows.  Label the internal lattice points in $D$ by $1, 2, \dots, a+b-1$ from
southwest to northeast.  As in the construction of $\Ass(a,b)$, for any lattice point $P$ which
is the bottom of a north step of $D$,  consider the laser $\ell(P)$.
These lasers give a topological decomposition of the region between $D$ and
the line $y = \frac{a}{b} x$.  For $1 \leq i < j \leq a+b-1$, we say that 
$i \sim j$ in $\pi(D)$ if the  labels $i$ and $j$ belong to the same connected component 
(where we consider the labels $i$ and $j$ to lie just below their vertices).

\begin{figure}
\begin{center}
\includegraphics[scale=1]{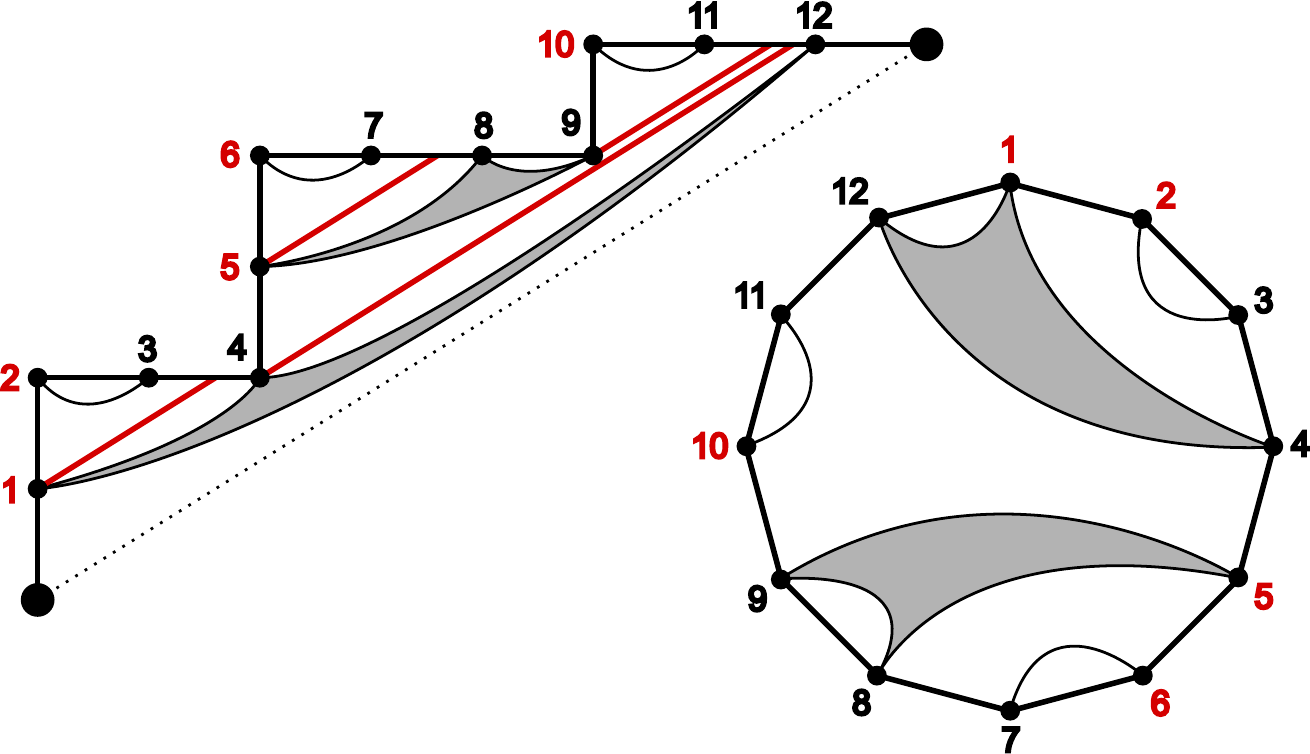}
\end{center}
\caption{A homogeneous noncrossing partition for $(a, b) = (5,8)$.}
\label{fig:homogeneous_nc}
\end{figure}

Figure~\ref{fig:homogeneous_nc} gives an example of this construction for
$(a, b) = (5, 8)$.  The internal lattice points of the Dyck path are labeled with 
$1, 2, \dots, 5+8-1 = 12$ and the relevant lasers are shown.  The resulting 
noncrossing partition of $[12]$ is drawn both on the Dyck path and on
a disk.  (The red indices will be deleted when we define inhomogeneous noncrossing
partitions in the next section.)

\begin{prop}
\label{basic-homogenous-nc-facts}
The set partition $\mu(D)$ of $[a+b-1]$ is noncrossing for any Dyck path $D$ and
the map $D \mapsto \mu(D)$ is injective.  Hence, there are $\Cat(a,b)$
homogeneous $(a,b)$-noncrossing partitions.
\end{prop}

\begin{proof}
It is topologically evident that the set partitions $\mu(D)$ are noncrossing.
It can be shown that the labels on the bottoms of the north steps of $D$ are the minimal labels
of the blocks of $\mu(D)$; the claim about injectivity follows.
\end{proof}

Figure~\ref{fig:rotpromo} shows $22$ of the $\frac{1}{13} {13 \choose 5} = 99$
$(a, b)$-homogeneous noncrossing partitions in the case $(a, b) = (5, 8)$ as 
$(5, 8)$-Dyck paths and as
set partitions of $[12]$.  
In Figure~\ref{fig:rotpromo} a Dyck path $D$ is drawn by leaving the cells to the northwest
of $D$ white and shading in the cells to the southeast of $D$.
These $22$ partitions are grouped together into orbits 
of the promotion operator on Dyck paths, to be defined in the next subsection.

In the classical case $(a, b) = (n, n+1)$, the homogeneous noncrossing partitions are precisely
the noncrossing matchings on $[2n]$.
In the Fuss case $(a, b) = (n, kn+1)$, the homogeneous noncrossing partitions are the
{\sf $(k+1)$-equal} noncrossing partitions on $[(k+1)n]$ (i.e., every block has size $k+1$).
This explains the adjective `homogeneous' in `homogeneous noncrossing partitions'.
As can be seen in Figure~\ref{fig:rotpromo},
homogeneous noncrossing partitions may have different block sizes in the general rational
case.

\subsection{Rotation and Promotion}
\label{Rotation and Promotion}
In the classical and Fuss cases, homogeneous noncrossing partitions are closed under the order
$(a+b-1)$ rotation operator.  It is natural to ask whether homogeneous rational noncrossing partitions
are also closed under rotation.  It turns out that they are, and we can describe the corresponding
action on Dyck paths explicitly. 

If $D$ is a Dyck path and if $P$ is an internal lattice point of $D$, let $t_P(D)$ be the Dyck path
defined as follows.  If $P$ is not a corner vertex, let $t_P(D) = D$.  If $P$ is a corner vertex, let
$t_P(D)$ be the lattice path obtained by interchanging the north and east steps on either side of $P$
(this turns $P$ from an outer corner to an inner corner, and vice versa), provided that this switch preserves the property of being a Dyck path (if it does not, set $t_P(D) = D$).  Define the 
{\sf promotion} operator $\rho$ on $(a,b)$-Dyck paths by
\begin{equation*}
\rho(D) = t_{P_{a+b-1}} \cdots t_{P_2} t_{P_1} (D),
\end{equation*}
where $P_i$ is the $i^{th}$ internal lattice point from the southeast of
$t_{P_{i-1}} \cdots t_{P_1}(D)$.  
Roughly speaking, the promotion image
 $\rho(D)$ is computed from $D$ by reading $D$ from southwest
to northeast and swapping corners of the form $NE$ and corners of the form $EN$
whenever possible.

\begin{prop}
\label{promotion-is-rotation}
The promotion operator on $(a,b)$-Dyck paths maps to counterclockwise rotation on
homogeneous $(a,b)$-noncrossing partitions.  In particular, homogeneous noncrossing partitions
are closed under rotation and $\rho^{a+b-1}$ is the identity operator on Dyck paths.
\end{prop}

Figure~\ref{fig:rotpromo} shows three orbits of the promotion and rotation operators when
$(a, b) = (5, 8)$.  The top orbit has size $3$, the middle orbit has size $6$, and the bottom
orbit has size $12$.

\begin{proof}
We find it convenient to give a more global description of promotion acting on an $(a,b)$-Dyck path $D$.  Interpret the path $D$ as tracing out an order ideal $I$, where boxes to the south-east of $D$ are in $I$ and boxes to the north-west of $D$ are not. 
The ideals $I$ are the shaded boxes in Figure~\ref{fig:rotpromo}.

Given a Dyck path $D$ with ideal $I$,
let $j$ be the west-most column of $D$ which contains no boxes in $I$, or $\infty$
if every column of $D$ contains boxes in $I$.
Then $I$ breaks naturally into two pieces: $I_W$, containing those boxes to the west of column $j$, 
and $I_E$, containing those boxes to the east of column $j$.  
Let $\rho(I)$ be defined by shifting $I_W$ one unit south 
(discarding any boxes to the south that would find themselves outside the allowed region) 
while shifting $I_E$ one unit west (appending new boxes to the south and to the east so 
that the resulting configuration is an order ideal).  By interpreting $\rho(I)$ as a Dyck path, 
it is easy to see that this is equivalent to the definition of the promotion $\rho(D)$ of the Dyck path.

We must now check that this rotates the homogenous $(a,b)$-noncrossing partition corresponding to $D$.  
Let $k$ be the label of the south-easternmost internal lattice point of the west-most empty column $j$---or, if there was no empty column, then let $k=a+b-1$.   
As the description of promotion given above preserves the 
relative positions of internal lattice points of $D$ within $I_E$ and $I_W$, 
we conclude from the laser construction  
of the $(a,b)$-noncrossing partition that those blocks 
with all labels greater than $k$ and those blocks with all labels less than $k$ are rotated by one.  
On the other hand, by considering the laser originating at the lattice point labeled $1$,
 $k$ is the smallest number that occurs in the same block as the label $1$.  
In particular, since the blocks define a noncrossing partition, this means that the 
only block that can contain both labels less than $k$ and labels greater than $k$ is the 
block containing $1$.  

The last verification we must perform, then, is that the block containing $1$ is rotated correctly.  
All of the labels of the block containing $1$---other than $1$ and $k$---are greater than $k$.  
After promotion, the step originating at $k$ is sent to a north step originating at $k-1$.  
If the step originating at $k$ is already north, then after shifting left, all of the labels previously visible to $k$ remain visible to $k-1$ and $a+b-1$ becomes visible, so that the block is rotated by one. 
 If the step originating at $k$ is east, let $l$ be the first label greater than $k$ such that the step originating at $l$ is north (if $k=a+b-1$, then let $l=a+b-1$). 
  Then $k,k+1,\ldots,l$ were in the same block as $1$; after promotion, the laser emanating from $k-1$ ensures that $k,\ldots,l-1$ are still in the same block, and the laser from $l-1$ ensures that the rest of the block is also correctly rotated (with $a+b-1$ becoming visible because of the shift west).

We conclude that all blocks are rotated by one, so that promotion of an $(a,b)$-Dyck path corresponds to rotation of the corresponding homogenous noncrossing partition.
\end{proof}

\begin{figure}
\begin{center}
\includegraphics[scale=0.4]{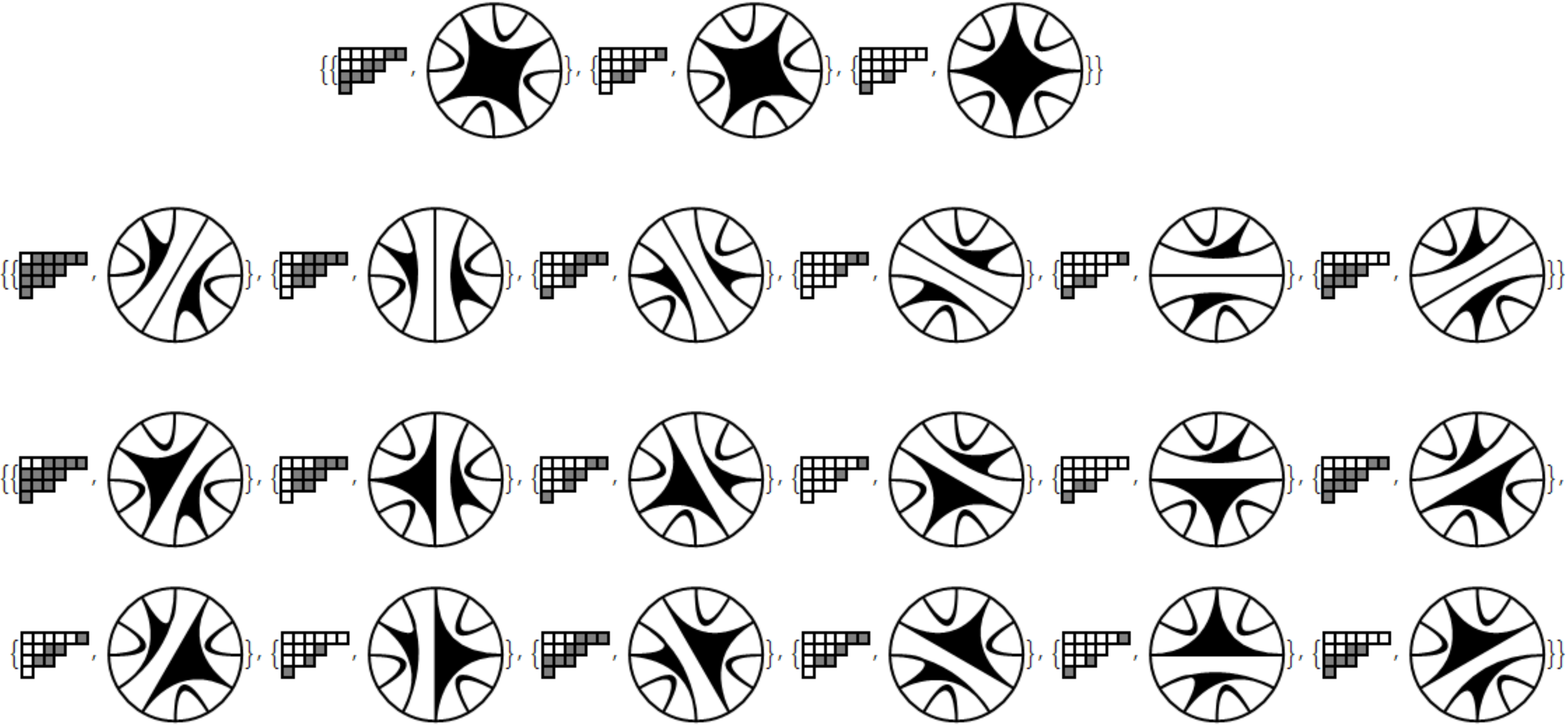}
\end{center}
\caption{Rotation and promotion on noncrossing matchings for $(a, b) = (5,8)$.}
\label{fig:rotpromo}
\end{figure}

We make the following conjecture about the cycle structure of $\rho$ on the set of $(a,b)$-Dyck
paths.  Recall that a triple $(X, C, X(q))$ is said to {\sf exhibit the cyclic sieving phenomenon} 
if $C = \langle c \rangle$ is a finite cyclic group acting on a finite set $X$,
$X(q) \in \NN[q]$ is a polynomial, and for all 
$d \geq 0$, we have that
\begin{equation*}
|X^{c^d}| = | \{ x \in X \,:\, c^d.x = x \} | = X(\omega^d),
\end{equation*}
where $\omega \in \CC$ is a primitive $|C|^{th}$ root of unity.  We use the standard $q$-analog 
notation $[m]_q = \frac{1-q^m}{1-q}$, $[m]!_q := [m]_q [m-1]_q \cdots [1]_q$, and
${m+n \brack m, n}_q := \frac{[m+n]!_q}{[m]!_q [n]!_q}$.  The following cyclic sieving conjecture
has been verified in the  case $b \equiv 1$ (mod $a$).

\begin{conj}
Let $X$ be the set of $(a,b)$-Dyck paths and let $C = \ZZ_{a+b-1} = \langle \rho \rangle$ 
act on $X$ by promotion.  The
triple $(X, C, X(q))$ exhibits the cyclic sieving phenomenon, where
$X(q) = \frac{1}{[a+b]_q} {a+b \brack a,b}_q$ is the $q$-analog of the rational Catalan number
$\Cat(a,b)$.
\end{conj}

\section{Rational Noncrossing Partitions}
\label{Rational Noncrossing Partitions}

\subsection{Construction}
\label{Construction}
The $(a, b)$-analog of rational 
noncrossing partitions will form a subset of the collection of classical
noncrossing partitions of $[b-1]$.  

Let $D$ be an $(a,b)$-Dyck path.  Label the right ends of the east steps of $D$
(besides the terminal lattice point $(a, b)$) with the labels $1, 2, \dots, (b-1)$.
For every {\bf valley} $P$ of the path $D$, fire the laser $\ell(P)$.  These valley 
lasers give a topological decomposition of the region between $D$ and the line
$y = \frac{a}{b} x$.  Define a partition $\pi(D)$ of $[b-1]$ by saying that 
$i \sim j$ in $\pi(D)$ if and only if the labels $i$ and $j$ belong to the same 
connected component (where labels, as before, lie just below their vertices).
Observe that we only fire lasers from valleys in this construction.  
The partition $\pi(D)$ is called an {\sf inhomogeneous $(a,b)$-noncrossing partition}.
Figure~\ref{fig:inhomogeneous_nc} gives an example of an inhomogeneous 
$(5,8)$-noncrossing
partition.

\begin{figure}
\begin{center}
\includegraphics[scale=1]{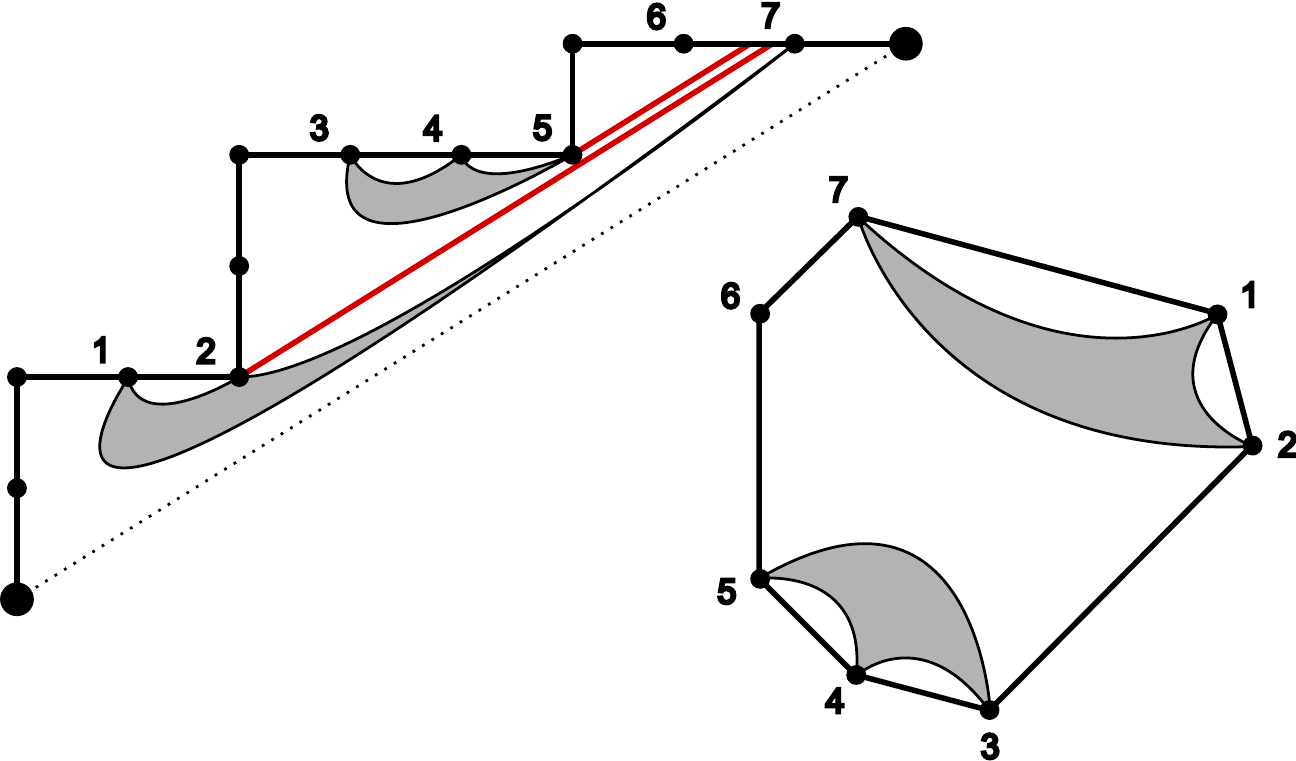}
\end{center}
\caption{An inhomogeneous noncrossing partition for $(a, b) = (5,8)$.}
\label{fig:inhomogeneous_nc}
\end{figure}

We understand inhomogeneous noncrossing partitions less well than their homogeneous
counterparts.
When $(a, b) = (n, n+1)$, the inhomogeneous
noncrossing partitions are the ordinary noncrossing partitions of $[n]$.  When 
$(a, b) = (n, kn+1)$, the inhomogeneous noncrossing partitions are exactly
the $k$-divisible noncrossing partitions of $[kn]$; that is, the noncrossing partitions of $[kn]$ whose
block sizes are all divisible by $k$.

\begin{prop}
\label{inhomogeneous-basic-facts}
1.  For any $(a, b)$-Dyck path $D$, the set partition $\pi(D)$ of $[b-1]$ is noncrossing.

2.  The number of blocks of $\pi(D)$ equals the number of vertical runs of $D$.

3.  The map $D \mapsto \pi(D)$ is injective, so there are $\Cat(a,b)$ inhomogeneous
$(a,b)$-noncrossing partitions.

4.  The collection of inhomogeneous $(a,b)$-noncrossing partitions forms an order filter
(i.e., an up-closed set)
in the lattice of all noncrossing partitions of $[b-1]$ given by $\mu \leq \tau$
if and only if $\mu$ refines $\tau$.
\end{prop}

\begin{proof}
Claim 1 is a topological consequence of the construction of $\pi(D)$ from the Dyck path $D$.
Claim 2 is also clear from the construction.

For Claim 3, observe that for any Dyck path
$D$ and any block $B$ of $\pi(D)$, the number of contiguous components of $B$
(as a subset of $1 < 2 < \dots < b-1$) equals the number of horizontal runs in $D$
on which the labels of $B$ lie.  
For example, for the partition in Figure~\ref{fig:inhomogeneous_nc}, the block
$\{1, 2, 7\}$ has two contiguous components and lies on two horizontal runs of the corresponding path, 
whereas the blocks
$\{3, 4, 5\}$ and $\{6\}$ both have a single contiguous component and lie on a single horizontal run.

We can use this observation to recursively construct 
$D$ from the partition $\pi(D)$.  Form a total order $B_1 < \dots < B_k$ on
the blocks of $\pi(D)$ 
by the rule $\mathrm{min}(B_1) < \dots < \mathrm{min}(B_k)$, where $\mathrm{min}(B)$ denotes
the minimal element of a block $B$.
Suppose that $D$ factors into nonempty north and east runs as 
$D = N^{i_1} E^{j_1} \cdots N^{i_k} E^{j_k}$.  If $D'$ is a Dyck path satisfying
$\pi(D) = \pi(D')$, by considering the fact that the block $B_k$ is the rightmost contiguous block
of $\pi(D')$, we conclude that $D'$ ends in $N^{i_k} E^{j_k}$.  
Similarly since $B_{k-1}$ is the rightmost contiguous block in the set partition obtained from
$\pi(D')$ by removing $B_k$,
the presence of the block
$B_{k-1}$ in $\pi(D')$ forces $D'$ to end in $N^{i_{k-1}} E^{i_{k-1}} N^{i_k} E^{j_k}$.
Continuing this process, we see that $D = D'$ and that the map $D \mapsto \pi(D)$ is 
injective.

The proof of Claim 4 is topological in nature.  Let $D$ be an $(a,b)$-Dyck path and let
$\pi'$ be a noncrossing partition of $[b-1]$ which covers $\pi(D)$ in the poset
of noncrossing partitions of $[b-1]$.  We want to show that there exists a Dyck path
$D'$ such that $\pi' = \pi(D')$.  We know that $\pi'$ is obtained from $\pi(D)$ by merging two
blocks of $\pi(D)$.  Call these merged blocks $B_1$ and $B_2$, so that
$\pi' = (\pi(D) \cup \{B_1 \cup B_2 \}) - \{B_1, B_2 \}$.
Without loss of generality $\min(B_1) < \min(B_2)$.

Let $\ell_1$ and $\ell_2$ be the lasers defining the `lower boundaries' of the regions corresponding
to $B_1$ and $B_2$ in $D$ (if $\min(B_1) = 1$, then we interpret $\ell_1$ to be the line
$y = \frac{a}{b} x$).  Let $P_1$ and $P_2$ be the lattice points on $D$ from which $\ell_1$ and $\ell_2$
are fired.  We form a new path $D'$ by moving the steps of the vertical run above $P_2$ in $D$
to the steps of the vertical run above $P_1$ in $D$.  Since $\ell_1$ is below $\ell_2$, the resulting
lattice path $D'$ is still an $(a,b)$-Dyck path.  We leave it to the reader to check that $\pi(D') = \pi'$. 
\end{proof}

As an example of Part 4 of Proposition~\ref{inhomogeneous-basic-facts} (and an illustration
of its proof), consider the $(5,8)$-inhomogeneous noncrossing partition
$\pi = \{1, 2, 7 / 3, 4, 5 / 6 \}$ shown in 
Figure~\ref{fig:inhomogeneous_nc}.  The set partition $\pi' := \{1, 2, 6, 7 / 3, 4, 5 \}$ covers
$\pi$ within the lattice of all noncrossing partitions of $[7]$.  The partition $\pi'$ was formed
from $\pi$ by merging the blocks $B_1 = \{1, 3, 7 \}$ and $B_2 = \{ 6 \}$.  The lasers
$\ell_1$ and $\ell_2$ defining the lower boundaries of $B_1$ and $B_2$ emanate from the lattice 
points $(0, 0)$ and $(5, 4)$ on the given Dyck path $D$.  To form a Dyck path
$D'$ giving rise to $\pi'$, we move the vertical run above $(5, 4)$ in $D$ (which consists
of a single north step) to the vertical run above $(0, 0)$ in $D$.  The resulting path
$D'$ is $D' = N^3 E^2 N^2 E^6$.  We leave it to the reader to verify that
$\pi(D') = \pi'$.

\subsection{Open Problems}
\label{Open Problems}

The inhomogeneous analog of Proposition~\ref{promotion-is-rotation} is still conjectural.

\begin{problem}
Prove that the inhomogeneous $(a,b)$-noncrossing partitions are closed under the rotation
action on $[b-1]$.  Describe the corresponding order $b-1$ operator on Dyck paths.
\end{problem}

More ambitiously, one could ask for a cyclic sieving phenomenon describing the action of rotation
on $(a,b)$-noncrossing partitions.    
The action of rotation on the level of Dyck paths seems hard to describe even in the classical
case $(a,b) = (n,n+1)$.
A possible method of attack would be to solve the following 
problem.

\begin{problem}
Give a nicer characterization of when a noncrossing partition of $[b-1]$ is a homogeneous
$(a,b)$-noncrossing partition.
\end{problem}

\section{Acknowledgements} 
The authors would like to thank Susanna Fishel, Stephen Griffeth, Jim Haglund, Mark Haiman, 
Brant Jones, Nick Loehr, Vic Reiner, Dennis Stanton, Monica Vazirani, 
and Greg Warrington for helpful discussions.

D. Armstrong was partially supported by NSF grant DMS - 1001825.
B. Rhoades was partially supported by NSF grant DMS - 1068861.

\end{document}